\documentclass[final,leqno]{amsart}
\usepackage{amsmath,amssymb,amsfonts,latexsym,stmaryrd}
\usepackage{graphicx,float,epsfig} 
\usepackage{color}

\DeclareGraphicsExtensions{ .eps,.ps} \usepackage{subfig}
\usepackage{multirow}
\usepackage{url}
\usepackage[linesnumbered,algoruled,boxed]{algorithm2e}
\usepackage[colorlinks=false, linkcolor=blue,  citecolor=OliveGreen,  pdfstartview={}{}]{hyperref} 
\usepackage[dvipsnames]{xcolor}

\def\ds{\displaystyle}
\def\O{\Omega}

\def\l{\lambda}

\def\t{\theta}

\renewcommand\sp{\mathop{\mathrm{Sp}}\nolimits}
\newtheorem{remark}{Remark}[section]
\newtheorem{lemma}{Lemma}[section]
\newtheorem{theorem}{Theorem}[section]
\newtheorem{prop}{Proposition}[section]

\newcommand{\set}[1]{\lbrace #1 \rbrace}

\newcommand{\norm}[1]{\lVert#1\rVert}

\newcommand{\n}{\boldsymbol{n}}
\usepackage{booktabs}
\usepackage{float}

\newcommand\bu{\boldsymbol{u}}
\newcommand\bv{\boldsymbol{v}}
\newcommand\bw{\boldsymbol{w}}
\newcommand\bx{\boldsymbol{x}}

\newcommand\bn{\boldsymbol{n}}

\newcommand\mE{\boldsymbol{\mathcal{E}}}

\newcommand\bF{\boldsymbol{f}}

\newcommand\bS{\boldsymbol{S}}
\newcommand\bT{\boldsymbol{T}}
\newcommand\bP{\boldsymbol{P}}

\newcommand\0{\mathbf{0}}

\newcommand\bxi{\boldsymbol{\xi}}

\def\CT{{\mathcal T}}

\newcommand\bcW{\mathbb{W}}

\newcommand\bcK{\boldsymbol{\mathcal{K}}}

\newcommand\bcI{\boldsymbol{\mathcal{I}}}

\newcommand\wbsig{\widetilde{\bsig}}

\newcommand\bsig{\boldsymbol{\sigma}}
\newcommand\btau{\boldsymbol{\tau}}

\newcommand\R{\mathbb{R}}

\renewcommand\H{\mathrm{H}}
\renewcommand\L{\mathrm{L}}

\renewcommand\O{\Omega}
\newcommand\DO{\partial\O}
\newcommand\bdiv{\mathop{\mathbf{div}}\nolimits}

\renewcommand\div{\mathop{\mathrm{div}}\nolimits}
\newcommand\rot{\mathop{\mathrm{rot}}\nolimits}
\newcommand\brot{\mathop{\mathbf{rot}}\nolimits}

\newcommand\tD{\mathtt{D}}
\newcommand\tr{\mathop{\mathrm{tr}}\nolimits}

\renewcommand\sp{\mathop{\mathrm{sp}}\nolimits}

\renewcommand\t{\mathtt{t}}

\newcommand\LO{\L^2(\O)}

\newcommand\HsO{\H^s(\O)}

\renewcommand\t{\mathtt{t}}
\newcommand\gap{\widehat{\delta}}

\begin{document}
\title[VEM approximation for the Stokes eigenvalue problem]
{A virtual element approximation for the pseudostress
formulation of the Stokes eigenvalue problem}

%\author{Ernesto C\'aceres}
%\address{Division of Applied Mathematics, Brown University, Providence, RI 02912, USA.}
%\email{ernesto\_caceres\_valenzuela@brown.edu}
%\thanks{The first  author was partially supported by Becas-CONICYT Programme for foreign students.}
\author{Felipe Lepe}
\address{Departamento de Matem\'atica, Universidad del B\'io-B\'io, Casilla 5-C, Concepci\'on, Chile.}
\email{flepe@ubiobio.cl}
\thanks{The first author was partially supported by
CONICYT-Chile through FONDECYT Postdoctorado project 3190204 (Chile).}

\author{Gonzalo Rivera}
\address{Departamento de Ciencias Exactas, Universidad de Los Lagos,
Casilla 933, Osorno, Chile.}
\email{gonzalo.rivera@ulagos.cl}  
\thanks{The second author was supported by
CONICYT-Chile through FONDECYT project 11170534 (Chile).}

%\author{XXX}
%\address{XXX,
%XXX.}
%\email{XXX@XXX.XXX}  
%\thanks{The third author was supported by
%XXXX}

%\author{Iv\'an Vel\'asquez}
%\address{CI$^{\mathrm{2}}$MA, Departamento de Ingenier\'{\i}a Matem\'atica,  Universidad de
%Concepci\'on, Casilla 160-C, Concepci\'on, Chile.}
%\email{ivelasquez@ci2ma.udec.cl}
%\thanks{The fourth author was partially supported by BASAL project CMM,
%Universidad de Chile (Chile).
%The fourth author was partially supported by CONICYT-Chile through
%the project AFB170001 of the PIA Program:
%Concurso Apoyo a Centros Cient\'ificos y Tecnol\'ogicos
%de Excelencia con Financiamiento Basal}

\subjclass[2000]{Primary 65N12, 65N15, 65N25, 65N30, 35Q35, 76D07}

\keywords{Spectral problems, Stokes eigenvalue problem,
virtual element method, error estimates.}
\begin{abstract}
In this paper we analyze a virtual element method (VEM) for a pseudostress formulation
of the Stokes eigenvalue problem.  This formulation allows to eliminate
the velocity and the pressure, leading to an elliptic formulation where the only unknown is the pseudostress tensor. The velocity and pressure can be recovered by a post-process. Adapting the non-compact operator theory, we prove that our method provides a correct approximation of the spectrum and is spurious free. We prove a priori error estimates, with optimal order,  which we confirm with some numerical tests.

\end{abstract}
\maketitle
\section{Introduction}\label{section:1}
Let $\O\subset\mathbb{R}^2$ be an open bounded domain with
Lipschitz boundary $\partial\O$. We assume that this boundary is 
splitted in two parts $\Gamma_D$ and $\Gamma_N$ such that
$\partial\O:=\Gamma_D\cup\Gamma_N$. We are interested
in the Stokes eigenvalue problem (see \cite{lovadina} for instance)
\begin{equation}\label{system_stokes}
\begin{split}
-\bdiv(\nabla\bu)+\nabla p&=\widehat{\lambda}\bu\hspace{0.9cm}\text{in}\,\O,\\
\div\bu &=0\hspace{1.2cm}\text{in}\,\O,\\
\bu&=\boldsymbol{0}\hspace{1.2cm}\text{on}\,\Gamma_D,\\
(\nabla\bu-p\mathbb{I})\bn&=\boldsymbol{0} \hspace{1.2cm}\text{on}\,\Gamma_N.
\end{split}
\end{equation}
where $\bu$ is the velocity, $p$ is the pressure, $\mathbb{I}$ is the identity matrix of $\mathbb{R}^{2\times 2}$ and $\boldsymbol{n}$ is the outward unitary 
vector on $\Gamma_N$. This problem its of much interest for mathematicians 
and engineers due the several applications in different fields, since 
the stability of fluids depends on the knowledge of the natural frequencies
of the Stokes spectral problem. 

It is well known that the classic velocity-pressure formulation like the analyzed in \cite{lovadina} has the advantage of approximate, for the two dimensional case for instance, 
three unknowns: the two components of the velocity and the scalar associated 
to the pressure. However, this mixed formulation is not suitable for the 
computational resolution when standard eigensolvers are used. 

On the other hand, the formulation analyzed in \cite{MMR3} where the so called pseudostress tensor is introduced,  leads to an elliptic problem where the only unknown is the mentioned tensor. Despite to fact that this formulation leads to approximate more unknowns compared with the velocity-pressure formulation, the resulting problem is elliptic and therefore, 
standard eigensolvers like $\texttt{eigs}$ of MATLAB works with no difficulties. Moreover, the velocity and the pressure of the Stokes eigenproblem can be recovered by postprocessing  the solution of the elliptic problem.

The pseudostress formulation has been recently analyzed in \cite{LM}, with different DG methods based in interior penalization. In this methods, the stabilization parameter affects
strongly the behavior of spurious eigenvalues and the choice of such parameter, in order to avoid the spurious eigenvalues,  depends on the configuration of the problem, namely the geometry and boundary conditions. On the other hand, the virtual element method (VEM), introduced in \cite{MR2997471},  results to be more attractive  since, in one hand, we are able to use arbitrary polygonal meshes and on the other, we don not have to deal with a penalization parameter and the extra terms related to  DG formulations. Also we remark the simplicity of the computational implementation of this method, compared with other classic finite element approximations.

In the present paper we introduce a high order VEM in order to solve problem \eqref{system_stokes} with the pseuodstress formulation introduced in \cite{MMR3}. Several papers deal with the Stokes and Navier Stokes problems, implementing the VEM in order to approximate the velocity and pressure considering different formulations (see for instance \cite{ABMVsinum14,BDV20,BLV17,BMV19,BDDD19,CG,CGM16,Clong19,DKS18,FM20,GMV19,GV18,GMS182,GMS18,IH19,VG18,CGMM20,MR3886352}). In particular, in \cite{CG} the authors analyze rigorously a VEM for the steady Stokes problem, introducing the pseudostress tensor which leads to a mixed formulation where the main unknowns are the velocity field and the pseudostress. By means of suitable VEM spaces and the corresponding projection operator, classic on the VEM setting, the authors show stability of the method and optimal order of approximation. Also in \cite{CGS} a mixed VEM is analyzed for the Brinkman problem. However, the analysis of these references are related to  source problems.   For our case, we will adapt the VEM framework developed in  \cite{CG, CGS}  for the eigenvalue problem formulation 
of \cite{MMR3}, where the virtual spaces and the corresponding virtual projection, designed for the tensorial source problem, will be 
useful for the spectral one. On the other hand, we have to deal with a non-compact solution operator in this pseudostress formulation, which implies the
adaptation of the classic theory of \cite{DNR1,DNR2} in the VEM setting, due the non conformity of the bilinear forms, in order to prove spectral correctness and error estimates.

The paper is organized as follows: In section \ref{section:2} we introduce the pseudostress formulation of problem \eqref{system_stokes} and recall basic properties of the corresponding solution operator of the spectral problem. In section \ref{sec:spec_app} we introduce the VEM framework where we will operate. This includes the standard hypothesis on the mesh, degrees of freedom, virtual spaces, approximation properties, and the discrete spectral problem of our interest. Section \ref{sec:spectral_app} is dedicated to the spectral analysis, namely the convergence and spurious free results. In section \ref{sec:error} we obtain error estimates for the eigenfunctions and eigenvalues  and finally, in section \ref{section:5}, we report some numerical tests which will confirm the theoretical results of our study.

We end this section with some of the notations that we will  use below. Given
any Hilbert space $X$, let $X^2$ and $\mathbb{X}$ denote, respectively,
the space of vectors and tensors  with
entries in $X$. In particular, $\mathbb{I}$ is the identity matrix of
$\R^{2\times 2}$ and $\mathbf{0}$ denotes a generic null vector or tensor. 
Given $\btau:=(\tau_{ij})$ and $\bsig:=(\sigma_{ij})\in\R^{2\times 2}$, 
we define as usual the transpose tensor $\btau^{\t}:=(\tau_{ji})$, 
the trace $\tr\btau:=\sum_{i=1}^2\tau_{ii}$, the deviatoric tensor 
$\btau^{\tD}:=\btau-\frac{1}{2}\left(\tr\btau\right)\mathbb{I}$, and the
tensor inner product $\btau:\bsig:=\sum_{i,j=1}^2\tau_{ij}\sigma_{ij}$. 

Let $\O$ be a polygonal Lipschitz bounded domain of $\R^2$ with
boundary $\DO$. For $s\geq 0$, $\norm{\cdot}_{s,\O}$ stands indistinctly
for the norm of the Hilbertian Sobolev spaces $\HsO$, $\HsO^2$ or
$\mathbb{H}^s(\O)$ for scalar, vectorial and tensorial fields, respectively, with the convention $\H^0(\O):=\LO$, $\H^0(\O)^{2}=\L^2(\O)^2$ and $\mathbb{H}^0(\O):=\mathbb{L}^2(\O)$. We also define for
$s\geq 0$ the Hilbert space 
$\mathbb{H}^{s}(\bdiv;\O):=\set{\btau\in\mathbb{H}^s(\O):\ \bdiv\btau\in\HsO^2}$, whose norm
is given by $\norm{\btau}^2_{\mathbb{H}^s(\bdiv;\O)}
:=\norm{\btau}_{s,\O}^2+\norm{\bdiv\btau}^2_{s,\O}$.
Henceforth, we denote by $C$ generic constants independent of the discretization
parameter, which may take different values at different places.

%%%%%%%%%%%%%%%%%%%%%%%%%%%%
\section{The continuous spectral problem}
\label{section:2}

We begin by recalling  the variational formulation
of the Stokes eigenvalue problem proposed in \cite{MMR3}
and some important results from this reference, which will be
needed for our analysis.
%
%The Stokes eigenvalue model problem of our interest is
%the following: find nontrivial $(\lambda,\bu,p)$ such that
%(see \cite{lovadina})

To study problem \eqref{system_stokes} we introduce the pseudostress tensor
$\bsig:=\nabla\bu-p\mathbb{I}$ (see \cite{cai2010,GMS1,GMS2}).
Then, we eliminate the pressure $p$ and the velocity $\bu$ 
(see \cite{MMR3} for further details),
to write the following eigenvalue problem
\begin{equation}\label{system2}
\begin{split}
-\bdiv\bsig&=\widehat{\lambda}\bu\hspace{0.9cm}\text{in}\,\O,\\
\bsig^{\tD}-\nabla\bu &=\boldsymbol{0} \hspace{1.2cm}\text{in}\,\O,\\
\bu&=\boldsymbol{0}\hspace{1.2cm}\text{on}\,\Gamma_D,\\
\bsig\bn&=\boldsymbol{0} \hspace{1.2cm}\text{on}\,\Gamma_N.
\end{split}
\end{equation}

We remark that the pressure can be recovered by the relation $p=-\frac{1}{2}\tr(\bsig)$. Then, using a shift argument, the variational formulation derived from \eqref{system2}  reads as follows: Find $\lambda\in\R$ and $\boldsymbol{0}\neq\bsig\in\bcW:=\{\btau\in\mathbb{H}(\bdiv;\O):\,
\btau\bn=\0\hspace{0.2cm}\text{on}\hspace{0.2cm}\Gamma_N\}$ such that
\begin{equation}\label{spect1}
a(\bsig,\btau)=\lambda b(\bsig,\btau)\qquad\forall\btau\in\bcW,
\end{equation}
where $\lambda:=1+\widehat{\lambda}$ and the bilinear forms $a:\bcW\times\bcW\rightarrow\mathbb{R}$
and $b:\bcW\times\bcW\rightarrow\mathbb{R}$ are defined as
\begin{align*}
\displaystyle a(\bsig,\btau)&:=\int_{\O}\bdiv\bsig\cdot\bdiv\btau+\int_{\O}\bsig^{\tD}:\btau^{\tD},\\
\displaystyle b(\bsig,\btau)&:=\int_{\O}\bsig^{\tD}:\btau^{\tD}.
\end{align*}
 
 The bilinear form $a(\cdot,\cdot)$ is $\bcW$-elliptic as stated in the following result.
 
%The following lemma gives the well-possedness of problem \eqref{spect1}.
\begin{lemma}\label{elipticCont}
There exists a constant $\alpha>0$, depending only on $\O$, such that
\begin{equation*}
 a(\btau,\btau)\geq\alpha\|\btau\|^2_{\mathbb{H}(\bdiv;\O)}\qquad\forall\btau\in\bcW.
\end{equation*}
\end{lemma}
\begin{proof}
See  \cite[Lemma 2.1]{MMR3}.
\end{proof}

Thanks to  Lemma~\ref{elipticCont}, we are in position to introduce
the solution operator $\bT$, defined as follows:
\begin{align*}
\bT:\bcW&\rightarrow\bcW,\\
           \bF&\mapsto \bT\bF:=\wbsig,
\end{align*}
where $\wbsig\in\bcW$ is the unique solution of the following source problem
\begin{equation*}\label{eq:source_prob}
 a(\wbsig,\btau)=b(\bF,\btau)\qquad\forall\btau\in\bcW.
\end{equation*}

As a consequence of Lax-Milgram lemma, we have that the linear operator $\bT$ is well defined and bounded.
Clearly the pair $(\kappa,\bsig)\in\mathbb{R}\times\bcW$ solves problem \eqref{spect1}
if and only if $(\mu=1/\kappa, \bsig)$ is an eigenpair of $\bT$,
with $\mu\neq 0$ and $\bsig\neq\boldsymbol{0}$. 
Moreover, the linear operator $\bT$ is self-adjoint
with respect to the inner product $a(\cdot,\cdot)$ in $\bcW$.

We introduce the following space
\begin{equation*}
\bcK:=\{\btau\in\bcW:\,\bdiv\btau=\0\,\,\text{in}\,\,\O\}.
\end{equation*}
It is clear that $\bT|_{\bcK}:\bcK\rightarrow\bcK$ reduces
to the identity, leading to the conclusion that $\mu=1$
is an eigenvalue of $\bT$ with associated eigenspace $\bcK$.

We recall from \cite{MMR3}  that there exists
an operator $\bP:\bcW\rightarrow\bcW$, defined as follows
\begin{align*}
\bP:\bcW&\rightarrow\bcW,\\
           \bxi&\mapsto \bP\bxi:=\widehat{\bsig},
\end{align*}
where $(\widehat\bsig,\widehat{\bu})\in\bcW\times \L^2(\O)^2$ is the solution of the following well posed mixed problem
\begin{equation}\label{eq:mixto}
\left\{
\begin{array}{cc}
\ds\int_{\O}\widehat{\bsig}^{\tD}:\btau^{\tD}+\int_{\O}\widehat{\bu}\cdot\bdiv\btau=0&\quad\forall\btau\in\bcW,\\
\ds\int_{\O}\bv\cdot\bdiv\widehat{\bsig}=\int_{\O}\bv\cdot\bdiv\bxi&\forall\bv\in \L^2(\O)^2,
\end{array}\right.
\end{equation}
which is the variational formulation of the following Stokes problem with external body force $-\bdiv\bxi$:
\begin{equation}\label{eq:system_div}
\begin{array}{rcl}
-\bdiv\widehat{\bsig}=-\bdiv\bxi&\quad\text{in}\,\O,\\
\widehat{\bsig}^{\tD}-\nabla\widehat{\bu}=\boldsymbol{0}&\quad\text{in}\,\O,\\
\widehat{\bu}=\boldsymbol{0}&\quad\text{on}\,\Gamma_D,\\
\widehat{\bsig}\boldsymbol{n}=\boldsymbol{0}&\quad\text{on}\,\Gamma_N.
\end{array}
\end{equation}
%\begin{align*}
%-\bdiv\wbsig=-\bdiv\bsig&\quad\text{in}\,\O,\\
%\wbsig^{\tD}-\nabla\tilde{\bu}=\boldsymbol{0}&\quad\text{in}\,\O,\\
%\tilde{\bu}=\boldsymbol{0}&\quad\text{on}\,\Gamma_D,\\
%\wbsig\boldsymbol{n}=\boldsymbol{0}&\quad\text{on}\,\Gamma_N.
%\end{align*}
Also, the solution $(\widehat{\bsig},\widehat{\bu})\in\bcW\times \L^2(\O)^2$ of problem \eqref{eq:mixto} satisfies the following 
estimate for $s\in(0,1]$ (see \cite[Lemma 3.2]{MMR3})
\begin{equation}\label{eq:estimate_P}
\|\widehat{\bsig}\|_{s,\O}+\|\widehat{\bu}\|_{1+s,\O}\leq C\|\bdiv\bxi\|_{0,\O}.
\end{equation}
Consequently, $\bP(\bcW)\subset \mathbb{H}^s(\O)$.
 
In summary the operator $\bP$ satisfies
the following properties:
\begin{itemize}
\item $\bP$ is idempotent and its kernel is given by $\bcK$;
\item There exist $C>0$ and $s\in(0,1]$ depending only on the geometry of $\O$ such that
$\bP(\bcW)\subset\mathbb{H}^s(\O)$ and $\Vert\bP(\btau)\Vert_{s,\O}\le C\Vert\bdiv\btau\Vert_{0,\O}$;
\item $\bP(\bcW)$ is invariant for $\bT$. Moreover, $\bP(\bcW)$ is orthogonal
to $\bcK$ with respect to the inner product $a(\cdot,\cdot)$ of $\bcW$.
\end{itemize}

As an immediate consequence of these properties, we have
that the space $\bcW$ is decomposed in the following direct
sum $\bcW=\bcK\oplus\bP(\bcW)$. Moreover, we have the following regularity
result, which proof follows the arguments of those in  \cite[Proposition~3.4]{MMR3}.

\begin{prop}\label{reg}
The operator $\bT$ satisfies
\begin{equation*}
\bT(\bP(\bcW))\subset\{\btau\in\mathbb{H}^s(\O):\bdiv\btau\in\H^{1+s}(\O)^2\},
\end{equation*}
and there exists $C>0$ such that, for all $\bF\in\bP(\bcW)$, if $\widetilde{\bsig}=\bT\bF$, then
\begin{equation*}
\norm{\widetilde{\bsig}}_{s,\O}+\norm{\bdiv\widetilde{\bsig}}_{1+s,\O}\leq C\norm{\bF}_{\mathbb{H}(\bdiv;\O)}, 
\end{equation*}
concluding that $\bT|_{\bP(\bcW)}:\bP(\bcW)\rightarrow \bP(\bcW)$ is compact.
\end{prop}

%\begin{remark}
%The previous lemma provides the inclusion $\bT(\bP(\bcW))\subset\bP(\bcW)\subset\H^s(\O)^{n\times n}$. 
%\end{remark}

With these results at hand, we have the following spectral characterization of operator $\bT$ proved in \cite[Theorem~3.5 ]{MMR3}.

\begin{lemma}\label{SPEC-CHAR}

The spectrum of  $\bT$ decomposes as follows: $\sp(\bT)=\{0,1\}\cup\{\mu_k\}_{k\in\mathbb{N}}$, where
\begin{itemize}
\item $\mu=1$ is an infinite-multiplicity eigenvalue of  $\bT$ and its associated eigenspace is $\bcK$;
\item $\mu=0$ is an eigenvalue of $ \bT$ and its associated eigenspace is 
\begin{equation*}
\boldsymbol{\mathcal{G}}:=\{\btau\in\bcW:\,\btau^{\tD}=\boldsymbol{0}\}=\{q\mathbb{I}:\, q\in\H^1(\O)\,\,\text{and}\,\, q=0\,\,\hbox{on}\,\,\Gamma_N\};
\end{equation*}
\item $\{\mu_k\}_{k\in\mathbb{N}}\subset (0,1)$ is a sequence of nondefective finite-multiplicity eigenvalues
of $\bT$ which converge to 0. 
\end{itemize}
\end{lemma}
The following result provides additional regularity  for the eigenfunction $\bsig$ associated to some eigenvalue $\mu\in(0,1)$.
\begin{prop}\label{additionalreg}
Let $\bsig\in\bcW$ be an eigenfunction associated with an eigenvalue $\mu\in (0,1)$. Then, there exists a positive constant $C>0$, depending on the eigenvalue, such that
\begin{equation*}
\norm{\bsig}_{r,\O}+\norm{\bdiv\bsig}_{1+r,\O}\leq C\norm{\bsig}_{\mathbb{H}(\bdiv;\O)}, 
\end{equation*}
with $r>0$.
\end{prop}
\begin{proof}
See  \cite[Proposition 2.2]{LM}.
\end{proof}

\section{Virtual Element Spectral Approximation}
\label{sec:spec_app}

In this section, we propose and analyze a virtual element method
to approximate the solutions of problem \eqref{spect1}. To do this task, we need to introduce some assumptions and definitions to operate in the virtual element setting.

\subsection{Construction and assumptions on the mesh}
%\label{subsec}

Let $\{\mathcal{T}_h(\O)\}_{h>0}$ be a sequence of decompositions of $\O$ into
elements $E$, %and let $\mathcal{E}_h$ be the set of edges $e$ of $\{\mathcal{T}_h(\O)\}_{h>0}$. 
We suppose that each $\{\mathcal{T}_h(\O)\}_{h>0}$ is built according with the procedure described below.

The polygonal domain $\Omega$ is partitioned
into a polygonal mesh $\mathcal{T}_h$ that is \emph{regular},
in the sense that there exist positive constants 
$ c, \eta$ such that
\begin{enumerate}
\item each edge $e \in \partial E$ has a length $h_e \ge c \: h_E$,
where $h_E$ denotes the diameter of $E$; 
\item each polygon $E$ in the mesh is star-shaped with respect to a ball of radius $\eta h_E$.
\end{enumerate}

For each integer $k\geq 0$ and for each $E\in\mathcal{T}_h$, we introduce the following local virtual element space of order $k$ (see \cite[Subsection 3.2]{CGS}):
\begin{multline}
\nonumber \mathbf{W}_h^E:=\{ \tau:=(\tau_1,\tau_2)^{\texttt{t}}\in\H(\div;E)\cap\H(\rot;E):\,\tau\cdot\boldsymbol{n}|_e\in \textsc{P}_k(e)\quad\forall e\subset\partial E,\\
 \quad\div\tau\in \textsc{P}_{k}(E), \quad\text{rot}\,\tau\in \textsc{P}_{k-1}(E) \},
\end{multline}
where $\text{rot}\tau:=\frac{\partial\tau_2}{\partial x_1}-\frac{\partial\tau_1}{\partial x_2}$ and $\textsc{P}_{-1}(E)=\{0\}$. Now, given $\tau\in \mathbf{W}_h^E$ we define the following degrees of freedom
\begin{align}
\ds\int_e\tau\cdot\boldsymbol{n} q\qquad&\forall q\in\textsc{P}_k(e)\quad\forall\,\text{edge}\,\,e\in\mathcal{T}_h,\label{eq:dof_normal0}\\
\int_E\tau\cdot\nabla q\qquad&\forall q\in\textsc{P}_{k}(E)\quad\forall E\in\mathcal{T}_h,\label{eq:dof_grad0}\\
\int_E\tau\cdot\boldsymbol{q}\qquad&\forall\boldsymbol{q}\in\mathcal{H}_{k}^{\bot}(E)\quad\forall E\in\mathcal{T}_h,\label{eq:dof_rot0} 
\end{align}
where $\mathcal{H}_k^{\bot}$ is a basis for $(\nabla\textsc{P}_{k+1}(E))^{\bot|_{\textbf{\textsc{P}}_k(E)}}\cap\textbf{\textsc{P}}_k(E)$, which is the $\textbf{\textsc{L}}^2$-orthogonal of $\nabla\textsc{P}_{k+1}(E)$ in $\textbf{\textsc{P}}_k(E)$. A complete description of the details and properties of these spaces can be found in \cite[Subsection 3.2]{CGS}. 

%On the other hand,  the amount of local degrees offreedom defined in \eqref{eq:dof_normal1}--\eqref{eq:dof_q} is given by
%\begin{equation*}
%n_k^E:=(k+1)\#(\mathcal{E}_h)+\left\{\frac{(k+1)(k+2}{2}-1 \right\}+\frac{k(k+1)}{2}=(k+1)(\#(\mathcal{E}_h)+k+1)-1.
%\end{equation*}}

We  now  introduce   for each $E\in\mathcal{T}_h$ the tensorial local virtual element space
\begin{equation}
\label{eq:global_space}
\bcW_{h}^{E}:=\{\btau\in\mathbb{H}(\bdiv;E)\cap\mathbb{H}(\brot;E) : (\tau_{i1},\tau_{i2})^{t}\in \mathbf{W}_{h}^{E}\quad \forall i\in\{1,2\}\},
\end{equation}
which is unisolvent respect to the following degrees of freedom:
\begin{align}
\ds\int_e\btau\boldsymbol{n}\cdot\boldsymbol{q}\qquad&\forall\boldsymbol{q}\in\textbf{\textsc{P}}_k(e)\quad\forall\,\text{edge}\,\,e\in\mathcal{T}_h,\label{eq:dof_normal}\\
\int_E\btau:\nabla\boldsymbol{q}\qquad&\forall\boldsymbol{q}\in\textbf{\textsc{P}}_{k}(E)\quad\forall E\in\mathcal{T}_h,\label{eq:dof_grad}\\
\int_E\btau :\boldsymbol{\rho}\qquad&\forall\boldsymbol{\rho}\in\boldsymbol{\mathcal{H}}_{k}^{\bot}(E)\quad\forall E\in\mathcal{T}_h,\label{eq:dof_rot}
\end{align}
where 
\begin{equation*}
\ds \boldsymbol{\mathcal{H}}_k^{\bot}:=\left\{\begin{pmatrix}\mathbf{q}\\\boldsymbol{0}\end{pmatrix} \,:\mathbf{q}\in\mathcal{H}_k^{\bot}(E)\right\}\cup\left\{\begin{pmatrix}\boldsymbol{0}\\\mathbf{q}\end{pmatrix} \,:\mathbf{q}\in\mathcal{H}_k^{\bot}(E)\right\}.
\end{equation*}
Finally, for every decomposition $ \mathcal{T}_h$ of $\O$ into simple polygons  $E$, we define  the global virtual element space
$$\bcW_h:=\{\btau\in \bcW: \btau|_E\in \bcW_h^E\,\, \text{for all}\,\, E\in \mathcal{T}_h\},$$

\subsection{Discrete bilinear forms}
In what follows we will define computable bilinear forms in order to analyze and implement the virtual element method. To do this task, we will define the split the bilinear forms  $a(\cdot,\cdot)$ and $b(\cdot,\cdot)$ as follos %$a_h:\bcW_h\times \bcW_h\rightarrow \mathbb{R}$ and $b_h:\bcW_h\times \bcW_h\rightarrow \mathbb{R}$ as follows
\begin{align*}
\displaystyle a(\bsig,\btau)=\sum_{E\in\mathcal{T}_h}a^E(\bsig,\btau)&:=\sum_{E\in\mathcal{T}_h}\int_{E}\bdiv\bsig\cdot\bdiv\btau+b(\bsig,\btau),\\
\displaystyle b(\bsig,\btau)=\sum_{E\in\mathcal{T}_h}b^E(\bsig,\btau)&:=\sum_{E\in\mathcal{T}_h}\int_{E}\bsig^{\tD}:\btau^{\tD}.
\end{align*}
%\begin{align*}
%\displaystyle a_h(\bsig,\btau)=\sum_{E\in\mathcal{T}_h}a_h^E(\bsig,\btau)&:=\sum_{E\in\mathcal{T}_h}\int_{E}\bdiv\bsig\cdot\bdiv\btau+b_h(\bsig,\btau),\\
%\displaystyle b_h(\bsig,\btau)=\sum_{E\in\mathcal{T}_h}b_h^E(\bsig,\btau)&:=\sum_{E\in\mathcal{T}_h}\int_{E}\bsig^{\tD}:\btau^{\tD}.
%\end{align*}
We observe that the term  $\int_{E}\bdiv\bsig\cdot\bdiv\btau$ is explicitly computable with the degrees of freedom defined in \eqref{eq:dof_normal}--\eqref{eq:dof_rot}. On the other hand, we have the term
\begin{equation*}
\sum_{E\in\mathcal{T}_h}b^E(\bsig,\btau)=\sum_{E\in\mathcal{T}_h}\int_{E}\bsig^{\tD}:\btau^{\tD}, 
\end{equation*}
which is not explicitly computable with the defined degrees of freedom. To overcome this difficulty, we need to introduce suitable spaces where the elements of $\bcW_h$ will be projected. 

With this aim, we introduce some operators for the analysis of our virtual element method. Let $\mathcal{P}_{k}^h:\L^2(\O)^{2}\rightarrow \{\boldsymbol{q}\in \L^2(\O)^{2}\;\ \boldsymbol{q}|_{E}\in\textbf{\textsc{P}}_{k}(E)\quad \forall E\in \mathcal{T}_h\}$ be the orthogonal projector which, for $\tau\in\L^2(\O)^{2}$, is characterized by
\begin{equation*}
\ds\int_E\mathcal{P}_{k}^h(\tau)\cdot\boldsymbol{q}=\int_E\tau\cdot\boldsymbol{q}\quad\forall E\in\mathcal{T}_h, \quad\forall\boldsymbol{q}\in\textbf{\textsc{P}}_{k}(E).
\end{equation*}
Note that $\mathcal{P}_{k}^h(\bv)|_E=\mathcal{P}_{k}^h(\bv|_E)$. Moreover, $\mathcal{P}_{k}^h(\tau)$ is explicitly computable for every $\tau\in\mathbf{W}_h^E$ using only its degree of freedom  \eqref{eq:dof_normal0}--\eqref{eq:dof_rot0}. 

On the other hand, for   $\boldsymbol{q}\in \textbf{\textsc{P}}_{k}(E)$ we know that  there exist unique $\boldsymbol{q}^{\bot}\in(\nabla \textsc{P}_{k+1}(E)^{\bot|_{\textbf{\textsc{P}}_k(E)}}\cap \textbf{\textsc{P}}_{k}(E)) $ and $\widetilde{q}\in  \textsc{P}_{k+1}(E)$, such that $\boldsymbol{q}=\boldsymbol{q}{\bot}+\nabla\widetilde{q}$, (see \cite{{CGS}} for more details). Then 
\begin{equation*}
\ds\int_E\tau_{h}\cdot\boldsymbol{q}=\int_{E}\tau_{h}\cdot\boldsymbol{q}^{\bot}+ \int_{E}\tau_{h}\cdot\nabla\widetilde{q}=\int_{E}\tau_{h}\cdot\boldsymbol{q}^{\bot}-\int_{E}\widetilde{q}\div\tau+\int_{\partial E}\tau\cdot\n\widetilde{q}.
\end{equation*}
 Also, for $m\in\{0,1,\ldots,k+1\}$, this operator satisfies the following error estimate (see \cite{CG} for further details),
\begin{equation*}
\|\tau-\mathcal{P}_k^h\tau\|_{0,E}\leq C h_E^m|\tau|_{m,E}\quad\forall\tau\in\H^m(E)^{2}, \,\forall E\in\mathcal{T}_h.
\end{equation*}

Now, inspired by the analysis presented in \cite[Subection 4.1]{CGS}, for each $E\in\mathcal{T}_h$ we define  $\widehat{\Pi}_h ^E:=\boldsymbol{\mathcal{P}}_{k}^h:\mathbb{L}^2(E)\rightarrow\mathbb{P}_k(E)$ be the $\mathbb{L}^2(E)$-orthogonal projector, which satisfies the following properties%\FL{that is to say $\widehat{\Pi}_h ^E$ stands for the operator $\mathcal{P}_{k}$ acting along each row of a tensor in $\mathbb{L}^2(E)$.
%$$\int_{E}\widehat{\Pi}_h ^E(\btau):\boldsymbol{\rho}=\int_{E}\btau:\boldsymbol{\rho},$$
 %such that satisfies the following properties}
\begin{itemize}
\item[(A.1)] There exists a positive constant $C$, independent of $E$, such that
\begin{equation*}
\|\widehat{\Pi}_h^E(\btau)\|_{0,E}\leq\|\btau\|_{0,E}\quad\forall\btau\in\mathbb{H}(\bdiv;E),
\end{equation*}
\item[(A.2)] $\ds\int_E\big(\widehat{\Pi}_h^E\btau \big)^{\tD}:\big(\widehat{\Pi}_h^E\boldsymbol{\rho} \big)^{\tD}=\int_E\big(\widehat{\Pi}_h^E\btau \big)^{\tD}:\boldsymbol{\rho}^{\tD}$, for all $\btau,\boldsymbol{\rho}\in\mathbb{H}(\bdiv;E)$, and
\item[(A.3)] given an integer $0\leq m\leq k+1$, there exists a positive constant $C$, independent of $E$, such that 
\begin{equation*}
\| \btau-\widehat{\Pi}_h^E\btau\|_{0,E}\leq C h_E^m|\btau|_{m,E},
\end{equation*}
for all $\btau\in\mathbb{H}^m(E)$. %\GR{or at least for all $\btau\in\bcW_{\nabla\curl}^{m}(E)$ where
%\begin{equation*}
%\bcW_{\nabla\curl}^{m}(E):=\{ \boldsymbol{\xi}\in\mathbb{H}^m(\O)\,:\, \boldsymbol{\xi}^{\tD}=\nabla\curl\boldsymbol{w}\,\,\text{for some}\,\,\boldsymbol{w}\in\H^{r+2}(E)^{2} \}
%\end{equation*}}
\end{itemize}

From \cite[Section 4]{CG}, (A.1) and (A.3) are straightforward, meanwhile (A.2) follows from the fact that if $\boldsymbol{\rho}\in\mathbb{P}_k(E)$ it holds that $\boldsymbol{\rho}^{\tD}\in\mathbb{P}_k(E)$ and, for all $\boldsymbol{\rho}, \btau\in\mathbb{P}_k(E$), we have
\begin{align*}
\int_E\left(\widehat{\Pi}_k^E\btau\right)^{\tD}:\left(\widehat{\Pi}_k^E\boldsymbol{\rho}\right)^{\tD}:=\int_E\widehat{\Pi}_k^E\boldsymbol{\rho}:\left(\widehat{\Pi}_k^E\btau\right)^{\tD}=\int_E \boldsymbol{\rho}:\left(\widehat{\Pi}_k^E\btau\right)^{\tD}=\int_E\left(\widehat{\Pi}_k^E\btau\right)^{\tD}:\boldsymbol{\rho}^{\tD}.
\end{align*}

On the other hand, let $S^E(\cdot,\cdot)$ be any symmetric positive definite bilinear form that satisfies
\begin{equation}
\label{eq:s_stable}
c_0\int_E\btau_h:\btau_h\leq S^E(\btau_h,\btau_h)\leq c_1\int_E\btau_h:\btau_h\quad\forall\btau_h\in\bcW_h^E,
\end{equation}
where $c_0$ and $c_1$ are positive constants depending on the mesh assumptions. Then, for each element we define the bilinear form
\begin{equation*}
\displaystyle b_h ^E(\bsig_h,\btau_h):=\int_E\left(\widehat{\Pi}_h^E\bsig_h\right)^{\tD}:\left(\widehat{\Pi}_h^E\btau_h\right)^{\tD}+S^E\left(\bsig_h-\widehat{\Pi}_h^E
\bsig_h,\btau_h-\widehat{\Pi}_h^E
\btau_h\right),
\end{equation*}
for  $\bsig_h,\btau_h\in\bcW_h^E$ and, in a natural way,
\begin{align*}
%\displaystyle a_h(\bsig,\btau)=\sum_{E\in\mathcal{T}_h}a_h^E(\bsig,\btau)&:=\sum_{E\in\mathcal{T}_h}\int_{E}\bdiv\bsig\cdot\bdiv\btau+b_h(\bsig,\btau),\\
\displaystyle b_h(\bsig,\btau):=\sum_{E\in\mathcal{T}_h}b_h^E(\bsig,\btau),\qquad \bsig_h,\btau_h\in\bcW_h^E.
\end{align*}

%\FL{In particular, if we consider $\widehat{\Pi}_k^E:=\mathcal{P}_k^E:\mathbb{L}^2(E)\rightarrow\textbf{\textsc{P}}_k(E)$, i.e, the $\mathbb{L}^2(E)$-orthogonal projection, we have that the dimension of $\textbf{\textsc{P}}_k^E$ is given by (see \cite[equation (4.13)]{CGS})
%\begin{equation*}
%\dim\textbf{\textsc{P}}_k^E=(k+1)(2k+4).
%\end{equation*}
%
%This choice for our virtual projection operator satisfies assumptions (A.1), (A.2) and (A.3). In fact (A.1) and (A.3) are straightforward, meanwhile (A.2) follows from the fact that if $\boldsymbol{\rho}\in\textbf{\textsc{P}}_k(E)$ it holds that $\boldsymbol{\rho}^{\tD}\in\textbf{\textsc{P}}_k(E)$ and, for all $\boldsymbol{\rho}, \btau\in\textbf{\textsc{P}}_k(E$), we have
%\begin{align*}
%\int_E\left(\widehat{\Pi}_k^E\boldsymbol{\rho}\right)^{\tD}:\left(\widehat{\Pi}_k^E\btau\right)^{\tD}=\int_E\widehat{\Pi}_k^E\btau:\left(\widehat{\Pi}_k^E\boldsymbol{\rho}\right)^{\tD}=\int_E \btau:\left(\widehat{\Pi}_k^E\boldsymbol{\rho}\right)^{\tD}=\int_E\left(\widehat{\Pi}_k^E\boldsymbol{\rho}\right)^{\tD}: \btau^{\tD}.
%\end{align*}}

The following result states that bilinear form $b_h^E(\cdot,\cdot)$ is stable.
\begin{lemma}
\label{lmm:stab}
For each $E\in\mathcal{T}_h$ there holds
\begin{equation*}
b_h^E(\boldsymbol{\rho},\btau)= b^E(\boldsymbol{\rho},\btau)\quad \forall\boldsymbol{\rho}\in\mathbb{P}_k(E),\quad\forall\btau\in\bcW_h^E,
\end{equation*}
and there exist constants $\alpha_1,\alpha_2$, independent of $h$ and $E$, such that
\begin{equation*}
\alpha_1 b^E(\btau,\btau)\leq b_h^E(\btau,\btau)\leq \alpha_2 \left(\|\btau\|_{0,E}^{2}+\|\btau-\widehat{\Pi}_h^E\btau\|_{0,E}^{2}\right)\quad\forall\btau\in\bcW_h^E.
\end{equation*}
\end{lemma}
\begin{proof}
See \cite[Lemma 4.6]{CG}.
\end{proof}
\subsection{Discrete spectral problem}

Now we will introduce the discretization of problem \eqref{spect1} which reads as follows: Find $\lambda_h\in\mathbb{R}$ and $\boldsymbol{0}\neq\bsig_h\in\bcW_h$ such that
\begin{equation}
\label{eq:disc_prob}
a_h(\bsig_h,\btau_h)=\lambda_h b_{h}(\bsig_h,\btau_h)\quad\forall\btau_h\in\bcW_h,
\end{equation}
where $\lambda_h=1+\widehat{\lambda}_h$ and
\begin{equation*}
a_h(\bsig_h,\btau_h):=\int_\O\bdiv \bsig_h\cdot\bdiv\btau_h+b_h(\bsig_h,\btau_h)\qquad \forall \bsig_h,\btau_h\in\bcW_h.
\end{equation*}

The following result establishes that the bilinear form $a_h(\cdot,\cdot)$ is elliptic in $\mathbb{W}_h$.
\begin{lemma}
\label{lmm:ellip_h}
There exists a constant $\widehat{\alpha}>0$, independent of $h$, such that
\begin{equation*}
a_h(\btau,\btau)\geq\widehat{\alpha}\|\btau\|^2_{\mathbb{H}(\bdiv;\O)}\quad\forall\btau\in\bcW_h.
\end{equation*}
\end{lemma}
\begin{proof}
See \cite[Lemma 5.1]{CGS}.
\end{proof}
With this ellipticity result at hand, we are in position to introduce the discrete solution operator
\begin{equation*}
\label{operatordisc}
\begin{split}
\bT_h:\bcW&\rightarrow\bcW_h,\\
\bF&\mapsto \bT_h\bF:=\wbsig_h,
\end{split}
\end{equation*}
where $\wbsig_h\in\bcW_h$ is the unique solution,
as a consequence of Lemma \ref{lmm:ellip_h} and the Lax-Milgram lemma,
of the following discrete source problem:
\begin{equation*}\label{charcTDG}
a_h(\wbsig_h,\btau_h)=b_h(\bF,\btau_h)\qquad\forall\btau_h\in\bcW_h.
\end{equation*}

It is easy to check that $\bT_h$ is self-adjoint respect to $a_h(\cdot,\cdot)$. On the other hand, $(\lambda_h,\bsig_h)\in\mathbb{R}\times\bcW_h$ solves \eqref{eq:disc_prob} if and only if $(\mu_h=1/\lambda_h, \bsig_h)$ is an eigenpair of $\bT_h$.
Moreover, as a direct consequence, we have that $\bT_h$ is well defined and uniformly
bounded respect to $h$.

We define the space
\begin{equation*}
\bcK_h:=\bcK\cap\bcW_h=\{\btau_h\in\bcW_h\,:\,\,\bdiv\btau_h=\boldsymbol{0}\,\,\text{in}\,\,\O\},
\end{equation*}
which is the eigenspace associated to the eigenvalue $\mu_h=1$ of $\bT_h$ (\cite[Lemma 4.3]{MMR3}).

Let  $\bcI_k^h: \mathbb{H}^t(\O) \to \bcW_h$ be the tensorial version of the VEM-interpolation operator, 
which satisfies the following classical error estimate, see \cite[Lemma 6]{BeiraoVemAcoustic2017},
\begin{equation}\label{asymp0}
 \norm{\btau - \bcI_k^h \btau}_{0,\O} \leq C h^{\min\{t, k+1\}} \norm{\btau}_{t,\O} \qquad \forall \btau \in \mathbb{H}^t(\O), \quad t>1/2.
\end{equation}
Also,  for less regular tensorial fields we have the following estimate, see \cite[Theorem 3.16]{hiptmair}
\begin{equation}\label{asymp00}
 \norm{\btau - \bcI_k^h \btau}_{0,\O} \leq C h^t (\norm{\btau}_{t,\O}
 + \norm{\btau}_{\bdiv,\O})  \quad \forall \btau \in \mathbb{H}^t(\O)\cap \H(\div;\O) \quad t\in (0, 1/2].
\end{equation}
Moreover, the following commuting diagram property holds true, see \cite[Lemma 5]{BeiraoVemAcoustic2017}:
\begin{equation}\label{asympDiv}
 \norm{\bdiv (\btau - \bcI_k^h\btau) }_{0,\O} = \norm{\bdiv \btau - \mathcal{P}_{k}^h \bdiv \btau }_{0,\O} 
 \leq C h^{\min\{t, k\}} \norm{\bdiv\btau}_{t,\O},
\end{equation}
for  $\bdiv \btau \in \H^t(\O)^{2}$ and  $\mathcal{P}_{k}^h$ being  the $\LO^2$-orthogonal projection onto $\textsc{P}_{k}$.
Also  we define the local restriction of the interpolant operator as $\btau_I:=\bcI_k^h(\btau)|_E\in \bcW_h^E$.

On the other hand, the discrete counterpart of operator $\bP$ is the operator $\bP_h:\bcW_{h}\rightarrow\bcW_h$, which satisfies  for $s\in(0,1]$, the following error estimate (see \cite[Lemma 4.4]{MMR3})
\begin{equation}
\label{eq:P-Ph}
\|\bP\bsig_h-\bP_h\bsig_h\|_{\mathbb{H}(\bdiv;\O)}\leq Ch^s\|\bsig\|_{\mathbb{H}(\bdiv;\O)}\quad\forall\bsig_h\in\bcW_h.
\end{equation}
Moreover, we have that $\bP_h|_{\bcW_h}$ is idempotent and $\bcW_h=\bcK_h\oplus\bP_h(\bcW_h)$. 

%The following spectral characterization of $\bT_h$
%holds.
%\begin{theorem}
%\label{thm:char_disc}
%The spectrum of $\bT_h$ consists of $M:=\dim(\bcW_h)$ nondefective eigenvalues, 
%repeating according their respective multiplicities. The spectrum decomposes as follows: $\sp(\bT_h)=\{0,1\}\cup\{\mu_{h}^{k}\}_{k=1}^K$. Moreover
%\begin{enumerate}
%\item the eigenspace associated with $\mu_h=1$ is $\bcK_h$;
%\item the eigenspace associated with $\mu_h=0$ is $\boldsymbol{\mathcal{G}}_h:=\{\btau_h\in\bcW_h\,:\,\,\btau_h^{\tD}=\boldsymbol{0}\}$;
%\item the eigenspaces associated with $\mu_{h}^{k}\in(0,1)$, $k=1,\ldots, K:=M-\dim(\bcK_h)-\dim(\boldsymbol{\mathcal{G}}_h)$ lie on $\bP_h(\bcW_h)$.
%\end{enumerate}
%\end{theorem}
%\begin{proof}
%See \cite[Theorem 4.5]{MMR3}.
%\end{proof}
\section{Spectral approximation}
\label{sec:spectral_app}
We begin this section by recalling some definitions of spectral theory. Let $\mathcal{X}$ be a generic Hilbert space and let $\bS$ be a linear bounded operator defined by $\bS:\mathcal{X}\rightarrow\mathcal{X}$. If $\boldsymbol{I}$ represents the identity operator, the spectrum of $\bS$ is defined by $\sp(\bS):=\{z\in\mathbb{C}:\,\,(z\boldsymbol{I}-\bS)\,\,\text{is not invertible} \}$ and the resolvent is its complement $\rho(\bS):=\mathbb{C}\setminus\sp(\bS)$. For any $z\in\rho(\bS)$, we define the resolvent operator of $\bS$ corresponding to $z$ by $R_z(\bS):=(z\boldsymbol{I}-\bS)^{-1}:\mathcal{X}\rightarrow\mathcal{X}$. 

Also, if $\mathcal{X}$ and $\mathcal{Y}$ are vectorial fields, we denote by $\mathcal{L}(\mathcal{X},\mathcal{Y})$ the space of all the linear and bounded operators acting from $\mathcal{X}$ to $\mathcal{Y}$.

The goal of this section is to prove the convergence between the solution operators and hence, the 
corresponding spectrums. To do this task, we resort to the theory of non-compact operators developed on \cite{DNR1}. In order to do this, we introduce some further notations. Let $\bS:\bcW\rightarrow \bcW$ be a bounded linear operator. We define
\begin{equation*}
\|\bS\|_h:=\sup_{\boldsymbol{0}\neq\btau_h\in\bcW_h}\frac{\|\bS\btau_h\|_{\mathbb{H}(\bdiv;\O)}}{\|\btau_h\|_{\mathbb{H}(\bdiv;\O)}}.
\end{equation*}

Let $\boldsymbol{\mathcal{X}}$ and  $\boldsymbol{\mathcal{Y}}$ be two closed subspaces of $\bcW$. We define the gap
$\widehat{\delta}$ between these subspaces by
\begin{equation*}
\widehat{\delta}(\boldsymbol{\mathcal{X}}, \boldsymbol{\mathcal{Y}}):=\max\{\delta(\boldsymbol{\mathcal{X}},\boldsymbol{\mathcal{Y}}), \delta(\boldsymbol{\mathcal{X}}, \boldsymbol{\mathcal{Y}}) \},
\end{equation*}
where
\begin{equation*}
\displaystyle \delta(\boldsymbol{\mathcal{X}},\boldsymbol{\mathcal{Y}}):=\sup_{\bx\in\boldsymbol{\mathcal{X}}}\delta(\bx,\boldsymbol{\mathcal{Y}})\quad\text{with}\,\,\,\delta(\bx,\boldsymbol{\mathcal{Y}}):=\inf_{\boldsymbol{y}\in\boldsymbol{\mathcal{Y}}}\|\bx-\boldsymbol{y}\|_{\mathbb{H}(\bdiv;\O)}.
\end{equation*}

Our next task is to check  the following properties of the non-compact operators theory \cite{DNR1}:
\begin{itemize}
\item P1: $\|\bT-\bT_h\|_h\rightarrow 0$ as $h\rightarrow 0$;
\item P2: $\forall\btau\in\bcW$, $\lim_{h\rightarrow 0}\delta(\btau,\bcW_h)=0$.
\end{itemize}

Since P2 is immediate due \eqref{asymp0} and  \eqref{asympDiv}, we only prove property P1. We begin with 
the following approximation result.
\begin{lemma}
\label{lmm:aproxTTh1}
Let $\bF\in\boldsymbol{P}(\bcW)$. Then, there exists a constant $C>0$ such that 
\begin{equation*}
\|(\bT-\bT_h)\boldsymbol{f}\|_{\mathbb{H}(\bdiv;\O)}\leq Ch^s\|\boldsymbol{f}\|_{\mathbb{H}(\bdiv;\O)}.
\end{equation*}
\end{lemma}
\begin{proof}
Let $\boldsymbol{f}\in\boldsymbol{P}(\bcW)$ such that % be chosen as the solution of \eqref{eq:mixto}
 $\widetilde{\bsig}:=\bT\bF$ and $\widetilde{\bsig}_h:=\bT_h\bF$. 
 %Due the first equation of \eqref{eq:system_div}, we have that $\bdiv\bF=\bdiv\bsig$ and, applying  \eqref{eq:estimate_P}, we obtain 
%\begin{equation}\label{eq:bound_f}
%\|\bF\|_{s,\O}\leq C\|\bdiv\bF\|_{0,\O}.
%\end{equation} 
Let $\widetilde{\bsig}_I\in\bcW_h$. We have that
\begin{equation}
\label{eq:T-Th}
\|(\bT-\bT_h)\boldsymbol{f}\|_{\mathbb{H}(\bdiv;\O)}\leq \|\widetilde{\bsig}-\widetilde{\bsig}_{I}\|_{\mathbb{H}(\bdiv;\O)}+\|\widetilde{\bsig}_h-\widetilde{\bsig}_I\|_{\mathbb{H}(\bdiv;\O)}.
\end{equation}
Set $\btau_h=\widetilde{\bsig}_h-\widetilde{\bsig}_I$. Hence, from Lemma \ref{lmm:ellip_h} we have
\begin{align*}
\widehat{\alpha}\|\btau_h\|^2_{\mathbb{H}(\bdiv;\O)}&\leq a_h(\btau_h,\btau_h)=a_h(\widetilde{\bsig}_h,\btau_{h})-a_h(\widetilde{\bsig}_I,\btau_h)\\
                                                              &=b_h(\boldsymbol{f},\btau_h)-a_h(\widetilde{\bsig}_I,\btau_h)\\
                                                              &=b_h(\boldsymbol{f},\btau_h)-b(\boldsymbol{f},\btau_h)+a(\widetilde{\bsig},\btau_h)-a_h(\widetilde{\bsig}_I,\btau_h)\\
       &=b_h(\boldsymbol{f},\btau_h)-b(\boldsymbol{f},\btau_h) +\int_{\O}\bdiv(\widetilde{\bsig}-\widetilde{\bsig}_I)\cdot\bdiv\btau_h\\
       &+\int_{\O}\widetilde{\bsig}^{\tD}:\btau_h^{\tD}-\sum_{E\in\mathcal{T}_h}b_h^E(\widetilde{\bsig}_I,\btau_h)\\                                                        
                                                      &=b_h(\boldsymbol{f},\btau_h)-b(\boldsymbol{f},\btau_h)+\int_{\O}\widetilde{\bsig}^{\tD}:\btau_h^{\tD}\\
                                                      & -\left[\sum_{E\in\mathcal{T}_h}\left\{b_h^E\left(\widetilde{\bsig}_I-\widehat{\Pi}_h^E\widetilde{\bsig},\btau_h\right)
+\int_{E}\left(\widehat{\Pi}_h^E\widetilde{\bsig}\right)^{\tD}:\btau_h^{\tD}\right\} \right],                                                 
\end{align*}
where in the last equality we have used commuting diagram property \eqref{asympDiv} and Lemma \ref{lmm:stab}. 
Then
\begin{align}\nonumber
\widehat{\alpha}\|\btau_h\|^2_{\mathbb{H}(\bdiv;\O)}&\leq                                                         
                                                      b_h(\boldsymbol{f},\btau_h)-b(\boldsymbol{f},\btau_h)\\\label{cotabtau}
                                                      &-\left[\sum_{E\in\mathcal{T}_h}\left\{b_h^E\left(\widetilde{\bsig}_I-\widehat{\Pi}_h^E\widetilde{\bsig},\btau_h\right)
+\int_{E}\left(\left(\widehat{\Pi}_h^E\widetilde{\bsig}\right)^{\tD}-\widetilde{\bsig}^{\tD}\right):\btau_h^{\tD}\right\} \right].    
\end{align}
We observe that
\begin{align*}
\ds b_h(\boldsymbol{f},&\btau_h)-b(\boldsymbol{f},\btau_h)=\sum_{E\in\mathcal{T}_h}\left(b_h^E(\boldsymbol{f}-\widehat{\Pi}_h^E\boldsymbol{f},\btau_h)+\int_{E}\left(\left(\widehat{\Pi}_h^E\boldsymbol{f}\right)^{\tD}-\boldsymbol{f}^{\tD}\right):\btau_h^{\tD}\right).
\end{align*}
%
%\begin{align*}
%\ds b_h(\boldsymbol{f}_h,&\btau_h)-b(\boldsymbol{f}_h,\btau_h)=\sum_{E\in\mathcal{T}_h}\left(b_h^E(\boldsymbol{f}_h,\btau_h)-\int_{E}\boldsymbol{f}_h^{\tD}:\btau_h^{\tD}\right)\\
%&=\sum_{E\in\mathcal{T}_h}\left(\int_{E}\left(\widehat{\Pi}_h^E\boldsymbol{f}\right)^{\tD}:\left(\widehat{\Pi}_h^E\btau\right)^{\tD}+S^E\left(\boldsymbol{f}_h-\widehat{\Pi}_h^E\boldsymbol{f}, \btau_h-\widehat{\Pi}_h^E\btau_h,\right)-\int_{E}\boldsymbol{f}_h^{\tD}:\btau_h^{\tD}\right)\\
%&=\sum_{E\in\mathcal{T}_h}\left(\int_{E}\left(\widehat{\Pi}_h^E\boldsymbol{f}-\boldsymbol{f}\right)^{\tD}:\left(\widehat{\Pi}_h^E\btau\right)^{\tD}+S^E\left(\boldsymbol{f}_h-\widehat{\Pi}_h^E\boldsymbol{f}, \btau_h-\widehat{\Pi}_h^E\btau_h\right)\right).
%\end{align*}

Now, from Lemma \ref{lmm:stab}, (A.3) and Cauchy-Schwarz inequality,  we have
%\begin{align*}
%S^E\left(\boldsymbol{f}_h-\widehat{\Pi}_h^E\boldsymbol{f}, \btau_h-\widehat{\Pi}_h^E\btau_h\right)&\leq c_1\int_E
%\left(\widehat{\Pi}_h^E\boldsymbol{f}-\boldsymbol{f}\right)^{\tD}:\left(\btau-\widehat{\Pi}_h^E\btau\right)^{\tD}\\
%&\leq c_1\left\|\left(\widehat{\Pi}_h^E\boldsymbol{f}-\boldsymbol{f}\right)^{\tD}\right\|_{0,E}\|\btau_h^{\tD}\|_{0,E}\\
%&\leq C\left\|\widehat{\Pi}_h^E\boldsymbol{f}-\boldsymbol{f}\right\|_{0,E}\|\btau_h\|_{0,E}.
%\end{align*}
\begin{align*}
\ds b_h(\boldsymbol{f},\btau_h)-b(\boldsymbol{f},\btau_h)&\leq C\sum_{E\in\mathcal{T}_h}\left\|\widehat{\Pi}_h^E\boldsymbol{f}-\boldsymbol{f}\right\|_{0,E}\|\btau_h\|_{0,E}.
\end{align*}

On the other hand, from Lemma \ref{lmm:stab}, (A.3) and (A.1),  we obtain
\begin{align*}
b_h^E\left(\widetilde{\bsig}_I-\widehat{\Pi}_h^E\widetilde{\bsig},\btau_h\right)&\leq C\left(\|\widetilde{\bsig}_I-\widehat{\Pi}_h^E\widetilde{\bsig}\|_{0,E}\|\btau_h\|_{0,E}+\|\widetilde{\bsig}_I-\widehat{\Pi}_h^E\widetilde{\bsig}_{I}\|_{0,E}\|\btau_h\|_{0,E}\right)\\
&\leq C\left(2\|\widetilde{\bsig}_I-\widehat{\Pi}_h^E\widetilde{\bsig}\|_{0,E}+\|\widehat{\Pi}_h^E(\widetilde{\bsig}-\widetilde{\bsig}_{I})\|_{0,E}\right)\|\btau_h\|_{0,E}\\
&\leq C\left(\|\widetilde{\bsig}-\widehat{\Pi}_h^E\widetilde{\bsig}\|_{0,E}+\|\widetilde{\bsig}-\widetilde{\bsig}_{I}\|_{0,E}\right)\|\btau_h\|_{0,E}.
\end{align*}

Substituting the above estimates in \eqref{cotabtau}, from the \eqref{eq:s_stable} and Cauchy-Schwarz inequality we obtain
\begin{multline}
\nonumber
\widehat{\alpha}\|\btau_h\|^2_{\mathbb{H}(\bdiv;\O)}\leq C\sum_{E\in\mathcal{T}_h} \left(\left\|\widehat{\Pi}_h^E\boldsymbol{f}-\boldsymbol{f}\right\|_{0,E}+\|\widetilde{\bsig}_I-\widetilde{\bsig}\|_{0,E}\right.\\
\left.+\|\widetilde{\bsig}-\widehat{\Pi}_h^E\widetilde{\bsig}\|_{0,E}\right)\|\btau_h\|_{0,E}.                                                       
 \end{multline}
 Then, from \eqref{eq:T-Th} we derive
\begin{multline}
\nonumber\|(\bT-\bT_h)\boldsymbol{f}\|_{\mathbb{H}(\bdiv;\O)}\leq C\sum_{E\in\mathcal{T}_h} \left(\left\|\widehat{\Pi}_h^E\boldsymbol{f}-\boldsymbol{f}\right\|_{0,E}+\|\widetilde{\bsig}_I-\widetilde{\bsig}\|_{0,E}\right.\\
\nonumber\left.+\|\widetilde{\bsig}-\widehat{\Pi}_h^E\widetilde{\bsig}\|_{0,E}\right).
\end{multline}
Finally, the proof follows from \eqref{asymp0}, \eqref{asymp00}, the fact that $\boldsymbol{f}\in\boldsymbol{P}(\bcW)$ and satisfies problem  \eqref{eq:estimate_P}, with data $\bdiv\bsig$, the approximation properties of $\widehat{\Pi}_h^E$ and Proposition \ref{reg}.
\end{proof}

Now we are in position to establish property P1.
\begin{lemma}
\label{lmm:P1}
There exists a positive constant $C$, independent of $h$, such that
\begin{equation*}
\|\bT-\bT_h\|_h\leq Ch^s.
\end{equation*}
\end{lemma}
\begin{proof}
For any $\bF_{h}\in \bcW_h$ and following step by step the proof in \cite[Lemma 5.1]{MMR3} we have
\begin{align*}
\|(\bT-\bT_{h})\bF_{h}\|_{\mathbb{H}(\bdiv;\O)}\leq C\big(\|(\bP_{h}-\bP)\bF_{h}\|_{\mathbb{H}(\bdiv;\O)}+\|(\bT-\bT_{h})\bP \bF_{h}\|_{\mathbb{H}(\bdiv;\O)}\big).
\end{align*}
For the first term on the right hand side, we invoke \eqref{eq:P-Ph} to obtain
\begin{equation}
\label{eq:term1}
\|(\bP_{h}-\bP)\bF_{h}\|_{\mathbb{H}(\bdiv;\O)}\leq C h^s\|\bF_{h}\|_{\mathbb{H}(\bdiv;\O)},
\end{equation}
and for the second term, we apply Lemma \ref{lmm:aproxTTh1}, which delivers
\begin{equation}
\label{eq:term2}
\|(\bT-\bT_{h})\bP \bF_{h}\|_{\mathbb{H}(\bdiv;\O)} \leq Ch^s\|\bP \bF_{h}\|_{\mathbb{H}(\bdiv;\O)}\leq C\|\bF_{h}\|_{\mathbb{H}(\bdiv;\O)},
\end{equation}
where the last inequality is an implication of Proposition \ref{reg} for $\widetilde{\bsig}=\bT\bP\bF_{h}$. Hence, gathering \eqref{eq:term1} and \eqref{eq:term2} we conclude the proof.

\end{proof}
As a consequence of Lemma \ref{lmm:P1}, the following results corresponding to \cite[Lemma 1 and Theorem 1]{DNR1}  hold true.
\begin{lemma}
\label{lmm:resolv_cont}
Assume that P1 holds true. Let $F\subset\rho(\bT)$ be a closed set. Then, there exist $C>0$ and $h_0$ independent of $h$, such that for $h<h_0$
\begin{equation*}
\ds\sup_{\btau_h\in\bcW_h}\|R_z(\bT_h)\btau_h\|_{\mathbb{H}(\bdiv;\O)}\leq C\|\btau_h\|_{\mathbb{H}(\bdiv;\O)}\quad\forall z\in F.
\end{equation*}
\end{lemma}
\begin{theorem}
\label{thm:spurious_free}
Let $V\subset\mathbb{C}$ be an open set containing $\sp(\bT)$. Then, there exists $h_0>0$ such that $\sp(\bT_h)\subset V$ for all $h<h_0$.
\end{theorem}
The main consequence of the previous results is that the proposed numerical
method does not introduces spurious eigenvalues. Moreover, according to \cite[Section 2]{DNR1} we have the spectral convergence of $\bT_h$ to $\bT$ as $h$ goes to zero. In fact, if $\mu\in (0,1)$ is an isolated eigenvalue of $\bT$ with multiplicity $m$ and $\mathcal{C}$ is an open circle on the complex plane centered at $\mu$ with boundary $\partial\mathcal{C}$, we have that $\mu$ is the only eigenvalue of $\bT$ lying in $\mathcal{C}$ and $\partial\mathcal{C}\cap\sp(\bT)=\emptyset$. Moreover, from \cite[Section 2]{DNR1} we deduce that for $h$ small enough there exist $m$ eigenvalues $\mu_h^{1},\ldots, \mu_h^{m}$ of $\bT_h$ (according to their respective multiplicities) that lie in $\mathcal{C}$ and hence, the eigenvalues $\mu_h^{i}$, $i=1,\ldots, m$ converge to $\mu$ as $h$ goes to zero.
\begin{remark}
\label{rmrk:inverse_R}
As a consequence of Lemma \ref{lmm:resolv_cont}, there exists a constant $C>0$ that, for $h$ small enough,
\begin{equation*}
\|(z\boldsymbol{I}-\bT_h)\btau_h\|_{\mathbb{H}(\bdiv;\O)}\geq C\|\btau_h\|_{\mathbb{H}(\bdiv;\O)}\quad\forall\btau_h\in\bcW_h, \,\,\forall z\in\partial\mathcal{C}.
\end{equation*}
\end{remark}

\section{error estimates}
\label{sec:error}
The aim of this section  is to obtain error estimates for our numerical method.  To do this task, and since the solution operator $\bT$ is non-compact, we resort to the theory of \cite{DNR2}. 

 We introduce some 
notations and definitions. Let $\mE$ be the eigenspace associated to $\bT$ corresponding to 
$\mu$  and let $\mE_h$ be the  invariant eigenspace associated to $\bT_h$  corresponding
to $\mu_{h}^1,\ldots,\mu_{h}^m$.

Let $\boldsymbol{\mathcal{P}}_h:\mathbb{L}^2(\O)\rightarrow\bcW_h\hookrightarrow\bcW$ be the projector with range $\bcW_h$, defined by the relation
\begin{equation*}
a(\boldsymbol{\mathcal{P}}_h\btau-\btau,\bv_h)=0\quad\forall\bv_h\in\bcW_h.
\end{equation*}
We recall that $a(\cdot,\cdot)$ is an inner product on $\bcW$. Hence, $\|\boldsymbol{\mathcal{P}}_h\btau\|_{\mathbb{H}(\bdiv;\O)}\leq\|\btau\|_{\mathbb{H}(\bdiv;\O)}$.
We define $\widehat{\bT}_h:=\bT_h\boldsymbol{\mathcal{P}}_h:\bcW\rightarrow\bcW_h$. With this operator at hand, we prove the following result (cf. \cite[Lemma 1]{DNR1}).
\begin{lemma}
\label{lemma:resolvent2}
There exist $h_0>0$ and $C>0$ such that
\begin{equation*}
\|R_z(\widehat{\bT}_h)\|_{\mathcal{L}(\bcW,\bcW)}\leq C\quad\forall z\in\partial\mathcal{C}, \quad\forall h\leq h_0.
\end{equation*}
\begin{proof}
See \cite[Lemma 11]{BeiraoVemAcoustic2017}.
\end{proof}
\end{lemma}
We define the spectral projector associated to $\bT$ by
\begin{equation*}
\ds \boldsymbol{F}:=\frac{1}{2\pi i}\int_{\partial\mathcal{C}}R_z(\bT)dz,
\end{equation*}
and the projector of  $\widehat{\bT}_h$ relative to $\mu_{1h},\ldots,\mu_{m(h)h}$
by 
\begin{equation*}
\widehat{\boldsymbol{F}}_h:=\frac{1}{2\pi i}\int_{\partial\mathcal{C}}R_z(\widehat{\bT}_h)dz.
\end{equation*}
With these definitions at hand, and considering the fact that our bilinear forms are not conforming, we prove the following result.
\begin{lemma}
\label{lemma:F_T}
There exist positive constants $C$ and $h_0$ such that, for all $h<h_0$, the following estimates hold
\begin{equation*}
\|(\boldsymbol{F}-\widehat{\boldsymbol{F}}_h)|_{\mE}\|_{\mathbb{H}(\bdiv;\O)}\leq C\|(\bT-\widehat{\bT}_h)|_{\mE}\|_{\mathbb{H}(\bdiv;\O)}\leq Ch^{\min\{r,k\}}.
\end{equation*}
\end{lemma}
\begin{proof}
The first estimate is a direct consequence of \cite[Lemma 3]{DNR2}, together with 
Lemma \ref{lemma:resolvent2}. For the second estimate, let $\boldsymbol{f}\in\mE$ be such that $\bsig:=\bT\boldsymbol{f}$ and $\bsig_h:=\widehat{\bT}_h\boldsymbol{f}=\bT_h\boldsymbol{\mathcal{P}}_h\boldsymbol{f}$. We recall that
$\boldsymbol{f}\in \mathbb{H}^r(\O)$ with $r>0$. Hence, invoking the first Strang lemma (see \cite[Theorem 4.1.1]{ciarlet}) we have
\begin{align*}
\|\bsig-\bsig_h\|_{\mathbb{H}(\bdiv;\O)}\leq C&\left( \|\bsig-\boldsymbol{\mathcal{P}}_h\bsig\|_{\mathbb{H}(\bdiv;\O)}+\sup_{\btau\in\bcW_h}\frac{|b(\boldsymbol{\mathcal{P}}_h\bsig, \btau_h)-b_h(\boldsymbol{\mathcal{P}}_h\bsig,\btau_h)|}{\|\btau_h \|_{\mathbb{H}(\bdiv;\O)}}\right. \\      
&\left.+\sup_{\btau\in\bcW_h}\frac{|b(\boldsymbol{f}, \btau_h)-b_h(\boldsymbol{\mathcal{P}}_h\boldsymbol{f},\btau_h)|}{\|\btau_h \|_{\mathbb{H}(\bdiv;\O)}}\right).
\end{align*}

Following the proof of \cite[Lemma 12]{BeiraoVemAcoustic2017}, together with  Lemma \ref{lmm:stab}   we obtain the following estimates for the consistency terms
\begin{align*}
|b(\boldsymbol{\mathcal{P}}_h\bsig, \btau_h)-b_h(\boldsymbol{\mathcal{P}}_h\bsig,\btau_h)|&\leq C\left(\|\bsig-\boldsymbol{\mathcal{P}}_h\bsig\|_{0,\O}+\|\bsig-\widehat{\Pi}_h^E\bsig \|_{0,\O} \right)\|\btau_h \|_{\mathbb{H}(\bdiv;\O)},\\
|b(\boldsymbol{f}, \btau_h)-b_h(\boldsymbol{\mathcal{P}}_h\boldsymbol{f},\btau_h)|&\leq C\left(\|\boldsymbol{f}-\boldsymbol{\mathcal{P}}_h\boldsymbol{f}\|_{0,\O}+\|\boldsymbol{f}-\widehat{\Pi}_h^E\boldsymbol{f} \|_{0,\O} \right)\|\btau_h \|_{\mathbb{H}(\bdiv;\O)}.
\end{align*}
Thus, we have
\begin{align*}
\|\bsig-\bsig_h\|_{\mathbb{H}(\bdiv;\O)}\leq C&\left(\|\bsig-\boldsymbol{\mathcal{P}}_h\bsig\|_{0,\O}+\|\bsig-\widehat{\Pi}_h^E\bsig \|_{0,\O}\right.\\
&\left. \|\boldsymbol{f}-\boldsymbol{\mathcal{P}}_h\boldsymbol{f}\|_{0,\O}+\|\boldsymbol{f}-\widehat{\Pi}_h^E\boldsymbol{f} \|_{0,\O}\right),
\end{align*}
which, according to \cite[Lemma 12]{BeiraoVemAcoustic2017}, leads to
\begin{equation*}
\|\bsig-\bsig_h\|_{\mathbb{H}(\bdiv;\O)}\leq C\left(\eta_h+\gamma_h\right)\leq Ch^{\min\{r,k\}},
\end{equation*}
where
\begin{equation*}
\eta_h:=\gap(\mE,\bcW)\leq Ch^{\min\{r,k\}}\quad\text{and}\quad\gamma_h:=\sup_{\bw\in\mE}\frac{\|\bw-\widehat{\Pi}_h^E\bw\|_{0,\O}}{\|\bw\|_{\mathbb{H}(\bdiv;\O)}}\leq Ch^{\min\{r,k\}}.
\end{equation*}
This concludes the proof.
\end{proof}
Let $\mE_h$ be the invariant subspace of $\bT_h$ relative
to the eigenvalues $\mu_{h}^{1},\ldots,\mu_{h}^{m}$ converging to $\mu$. We have the following result.
\begin{lemma}
\label{lemma:op_lam}
Let
\begin{equation*}
\boldsymbol{\Lambda}_h:=\widehat{\bF}_h|_{\mE}:\mE\rightarrow\mE_h.
\end{equation*}
For $h$ small enough, the operator $\boldsymbol{\Lambda}_h$ is invertible and there exists $C$ independent of $h$ such that
\begin{equation*}
\|\boldsymbol{\Lambda}^{-1}\|_{\mathcal{L}(\bcW, \bcW)}\leq C.
\end{equation*} 
\end{lemma}
\begin{proof}
See \cite[Lemma 13]{BeiraoVemAcoustic2017}.
\end{proof}
Now we are in position to establish error estimates for the approximation of the eigenspaces. 
\begin{theorem}
\label{thm:error_spaces}
There exists $C>0$ such that
\begin{equation*}
\widehat{\delta}(\mE,\mE_h)\leq Ch^{\min\{r,k\}}.
\end{equation*}
\end{theorem}
\begin{proof}
The proof follows from Lemmas \ref{lemma:F_T} and \ref{lemma:op_lam}, and runs identically as in  \cite[Theorem 1]{DNR2}.
\end{proof}
We end this section with the following theorem which establishes the double order of convergence for
the eigenvalues. To this end, we note that the error estimate for the eigenvalue $\mu$ of $\bT$ leads to an analogous estimate for the approximation of the eigenvalue $\l= \dfrac{1}{\mu}$ of \eqref{spect1} with eigenspace $\mE$. Let  $\l_{h}^{(i)}=\dfrac{1}{\mu_{h}^{(i)}}$, $1\leq  i\leq m$ be the eigenvalues of \eqref{eq:disc_prob}  with invariant subspace $\mE_{h}$. Therefore we have the following result.
\begin{theorem}
\label{thm:double:order}
There exist positive constants $C$ and $h_0$, such that for $h<h_0$
\begin{equation*}
|\lambda-\lambda_h^{(i)}|\leq Ch^{2\min\{r,k\}},\qquad i=1,\ldots, m.
\end{equation*}
\end{theorem}
\begin{proof}
Let $\bsig_h\in\mE_h$ be an eigenfunction corresponding  to one of the eigenvalues
$\lambda_h^{(i)}$ with $i=1,\ldots, m$ and $\|\bsig_h\|_{\mathbb{H}(\bdiv;\O)}$. Since $\delta(\bsig_h,\mE)\leq Ch^{\min\{r,k\}}$, there exists $\bsig\in\mE$ such that
\begin{equation}
\label{eq:estima_div}
\|\bsig-\bsig_h\|_{\mathbb{H}(\bdiv;\O)}\leq Ch^{\min\{r,k\}}.
\end{equation}
Since $a(\cdot, \cdot)$, $b(\cdot, \cdot)$ and $b_h(\cdot, \cdot)$ are symmetric and
$\bsig$ and $\bsig_h$ solves \eqref{spect1} and \eqref{eq:disc_prob}, respectively, we have
\begin{align*}
a(\bsig-\bsig_h, \bsig-\bsig_h)-\lambda b(\bsig-\bsig_h,\bsig-\bsig_h)&=a(\bsig_h,\bsig_h)-\lambda b(\bsig_h,\bsig_h)\\
&=\l\boldsymbol{\Lambda}_1+\boldsymbol{\Lambda}_2+(\lambda_h^{(i)}-\lambda)b_{h}(\bsig_h,\bsig_h),
\end{align*}
where
\begin{equation*}
\boldsymbol{\Lambda}_1:=[b_h(\bsig_h,\bsig_h)-b(\bsig_h,\bsig_h)]\quad\text{and}\quad
\boldsymbol{\Lambda}_2:=[a(\bsig_h,\bsig_h)-a_h(\bsig_h,\bsig_h)].
\end{equation*}
Hence, we have the following identity
\begin{equation}
\label{eq:padra}
(\lambda_h^{(i)}-\lambda)b_{h}(\bsig_h,\bsig_h)=a(\bsig-\bsig_h,\bsig-\bsig_h)-\lambda b(\bsig-\bsig_h,\bsig-\bsig_h)-\l\boldsymbol{\Lambda}_1-\boldsymbol{\Lambda}_2, 
\end{equation}
where we need to estimate each of the contributions on the right hand side of \eqref{eq:padra}. We begin with the first two terms. 

Since $\bsig^{\tD}-\btau^{\tD}=(\bsig-\btau)^{\tD}$ and $\|\tr(\btau)\|_{0,\O}\leq\sqrt{2}\|\btau\|_{0,\O}$, we have
\begin{align}
\nonumber\left|\int_{\O}(\div(\bsig-\bsig_h))^2+(1-\lambda)\int_{\O}(\bsig^{\tD}-\bsig_h^{\tD})^2\right|\\
&\hspace{-1.5cm}\leq\|\bdiv(\bsig-\bsig_h)\|_{0,\O}^2+
|1-\lambda|\|\bsig^{\tD}-\bsig_h^{\tD}\|_{0,\O}^2\nonumber\\
&\hspace{-1.5cm}\leq C\|\bsig-\bsig_h\|_{\mathbb{H}(\bdiv;\O)}^2\leq Ch^{2\min\{r,k\}},\label{eq:estimaA}
\end{align}
where we have used \eqref{eq:estima_div}. Now using Lemma \ref{lmm:stab} we estimate $\boldsymbol{\Lambda}_1$ as follows
\begin{align*}
|b_h&(\bsig_h,\bsig_h)-b(\bsig_h,\bsig_h)|=\left|\sum_{E\in\mathcal{T}_h}\left(b_{h}^{E}(\bsig_h-\widehat{\Pi}_h^E\bsig_h,\bsig_h-\widehat{\Pi}_h^E\bsig_h)\right.\right.\\
&\hspace{3.5cm}\left.\left.-b^{E}(\bsig_h-\widehat{\Pi}_h^E\bsig_h,\bsig_h-\widehat{\Pi}_h^E\bsig_h)\right)\right|\\
%\ds&=\left|\sum_{E\in\mathcal{T}_h}\left(\int_E\left\{\left(\widehat{\Pi}_h^E\bsig_h\right)^{\tD}\right\}^2+S^E\left(\bsig_h-\widehat{\Pi}_h^E\bsig_h,\bsig_h-\widehat{\Pi}_h^E\bsig_h\right)\right)-\sum_{E\in\mathcal{T}_h}\int_E(\bsig_h^{\tD})^2\right|\\
&\leq\sum_{E\in\mathcal{T}_h}2\alpha_2\left\|\bsig_h-\widehat{\Pi}_h^E\bsig_h\right\|_{0,E}^{2}
+\sum_{E\in\mathcal{T}_h}\int_E\left\{\left(\bsig-\widehat{\Pi}_h^E\bsig_h \right)^{\tD}\right\}^2\\
&=2\alpha_2\sum_{E\in\mathcal{T}_h}\left\|\bsig_h-\widehat{\Pi}_h^E\bsig_h\right\|_{0,E}^2+\sum_{E\in\mathcal{T}_h}\left\|\left(\bsig_h-\widehat{\Pi}_h^E\bsig_h\right)^{\tD}  \right\|_{0,E}^2\\
& \leq C\left\| \bsig_h-\widehat{\Pi}_h^E\bsig_h\right\|_{0,\O}^2\\
&\leq C\left(\|\bsig-\bsig_h\|_{0,\O}^2+\left\|\bsig-\widehat{\Pi}_h^E\bsig\right\|_{0,\O}^2+\left\|\widehat{\Pi}_h^E(\bsig-\bsig_h)\right\|_{0,\O}^2 \right)\leq Ch^{2\min\{r,k\}},
\end{align*}
where we have used the definition of $b_h(\cdot,\cdot)$, the fact that $\widehat{\Pi}_h^E$ is a projection, \eqref{eq:s_stable} and \eqref{eq:estima_div}.

 We now  estimate $\boldsymbol{\Lambda}_2$. To do this task, we use the definition of each bilinear form, elementwise, as follows
 \begin{align*}
|a_h&(\bsig_h,\bsig_h)-a(\bsig_h,\bsig_h)|=|b_h(\bsig_h,\bsig_h)-b(\bsig_h,\bsig_h)|\\
& \leq C\left\| \bsig_h-\widehat{\Pi}_h^E\bsig_h\right\|_{0,\O}^2\\
&\leq C\left(\|\bsig-\bsig_h\|_{0,\O}^2+\left\|\bsig-\widehat{\Pi}_h^E\bsig\right\|_{0,\O}^2+\left\|\widehat{\Pi}_h^E(\bsig-\bsig_h)\right\|_{0,\O}^2 \right)\leq Ch^{2\min\{r,k\}},
 \end{align*}
 Hence, by  following the same steps that leads to the estimate of $\boldsymbol{\Lambda}_1$, we obtain that
 \begin{equation*}
 |a_h(\bsig_h,\bsig_h)-a(\bsig_h,\bsig_h)|\leq Ch^{2\min\{r,k\}}.
 \end{equation*}
 On the other hand, since $\lambda_h^{(i)}\rightarrow \lambda$ as $h$ goes to zero and Lemma \ref{lmm:ellip_h},  we have
 \begin{equation}
 \label{eq:estimaB}
 \ds b_h(\bsig_h,\bsig_h)\geq\frac{a_h(\bsig_h,\bsig_h)}{\lambda_h^{(i)}}\geq \widehat{\alpha}\frac{\|\bsig_h\|_{\mathbb{H}(\bdiv;\O)}^2}{\lambda_h^{(i)}}=\widehat{C}>0.
 \end{equation}
 Finally, gathering \eqref{eq:estimaA}, the bounds of $\boldsymbol{\Lambda}_1$ and $\boldsymbol{\Lambda}_2$, and \eqref{eq:estimaB}, we conclude the proof.
\end{proof}
\section{Numerical results}
\label{section:5}
In the following section we report numerical examples in order to asses the performance
of our numerical method. For all the experiments we have considered the lowest order  polynomials ($k=0$). 
We present tests in different domains
where we compute eigenvalues whit different polygonal meshes and orders of convergence. To do this task,  the computational domains that we will consider
are two different squares, each of them with different boundary conditions, and a L-shaped domain. 
All the reported results have been obtained with a MATLAB code. Also, in each table we show in the column 'Extr.', extrapolated values obtained with a least-square fitting which we compare with the values of some particular references located in the last column of every table.

In Figure \ref{FIG:meshes} we present the meshes that we will consider for our tests.

%\vspace{-2cm}
\begin{figure}[H]
	\begin{center}
		\begin{minipage}{13cm}
			\centering\includegraphics[height=3cm, width=3cm]{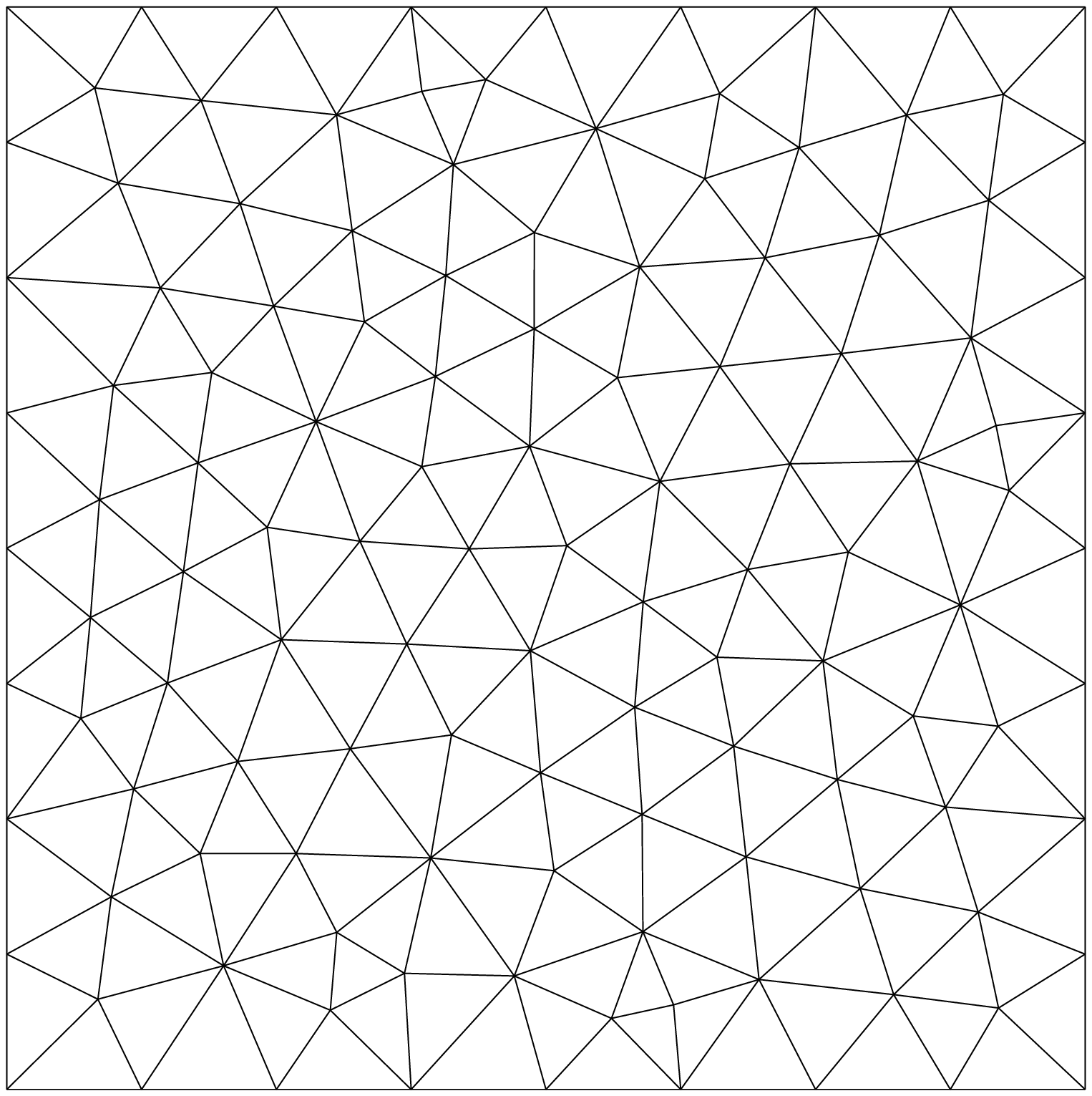}
			\centering\includegraphics[height=3cm, width=3cm]{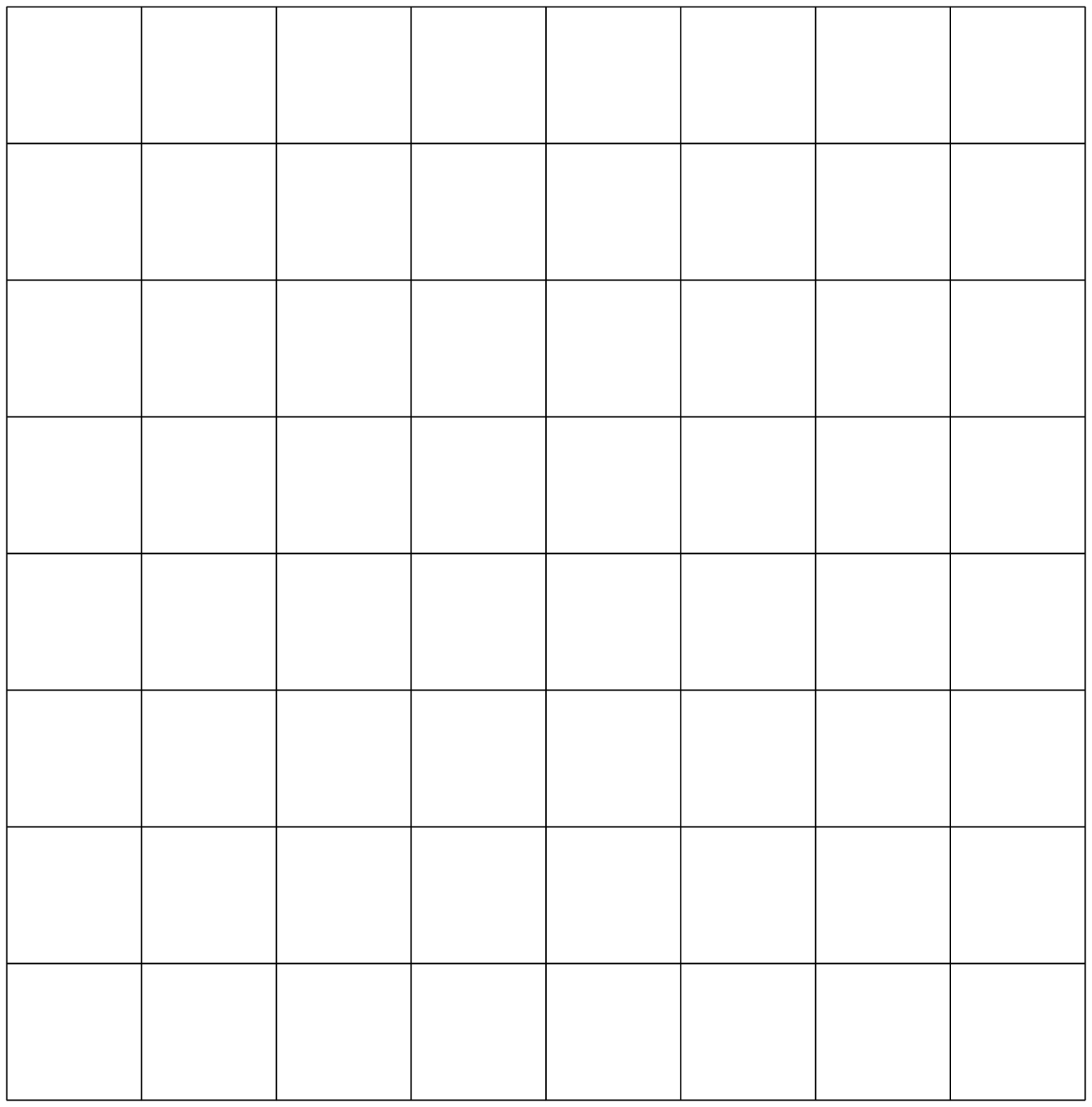}
                          \centering\includegraphics[height=3cm, width=3cm]{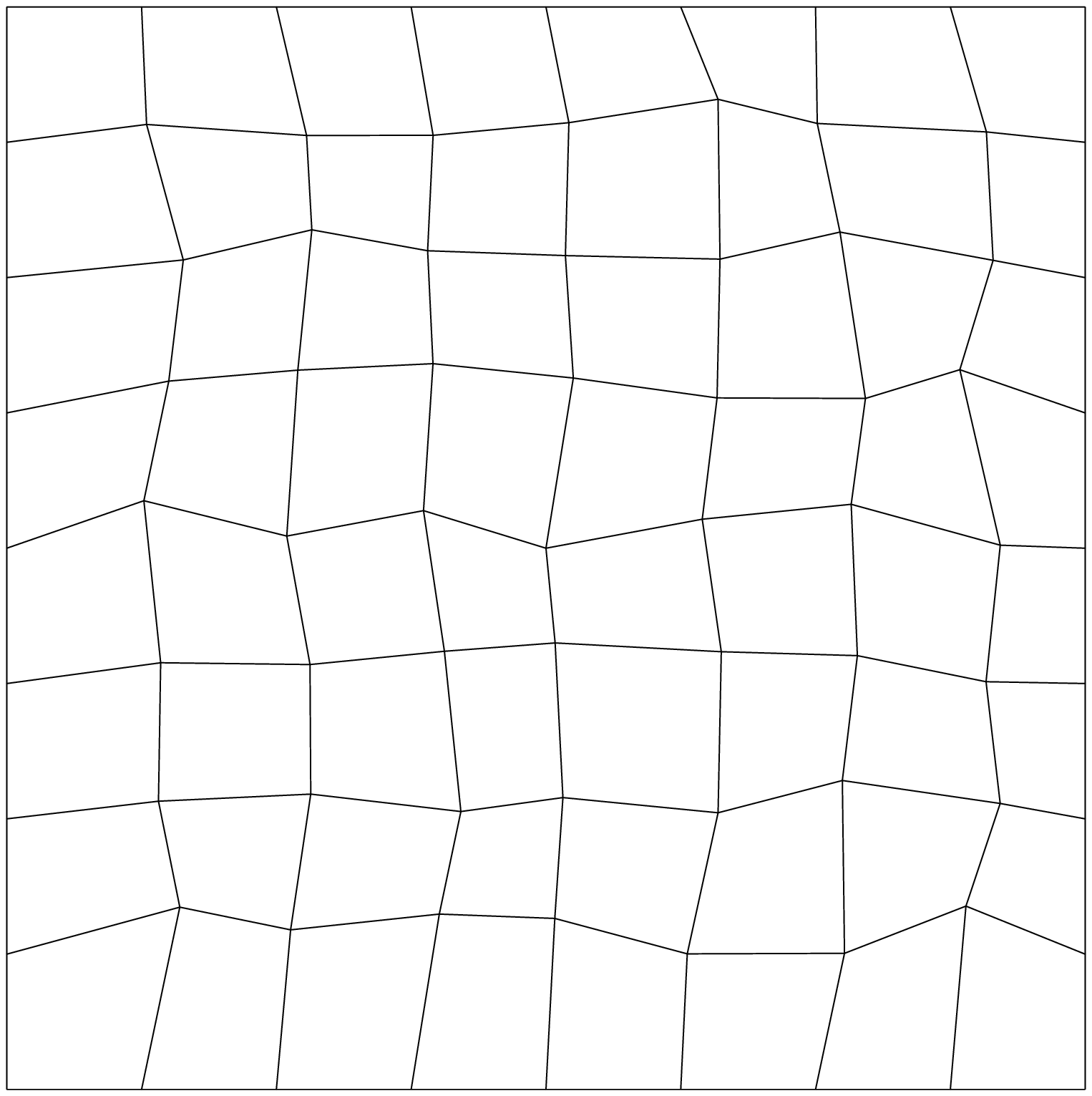}
                          \centering\includegraphics[height=3cm, width=3cm]{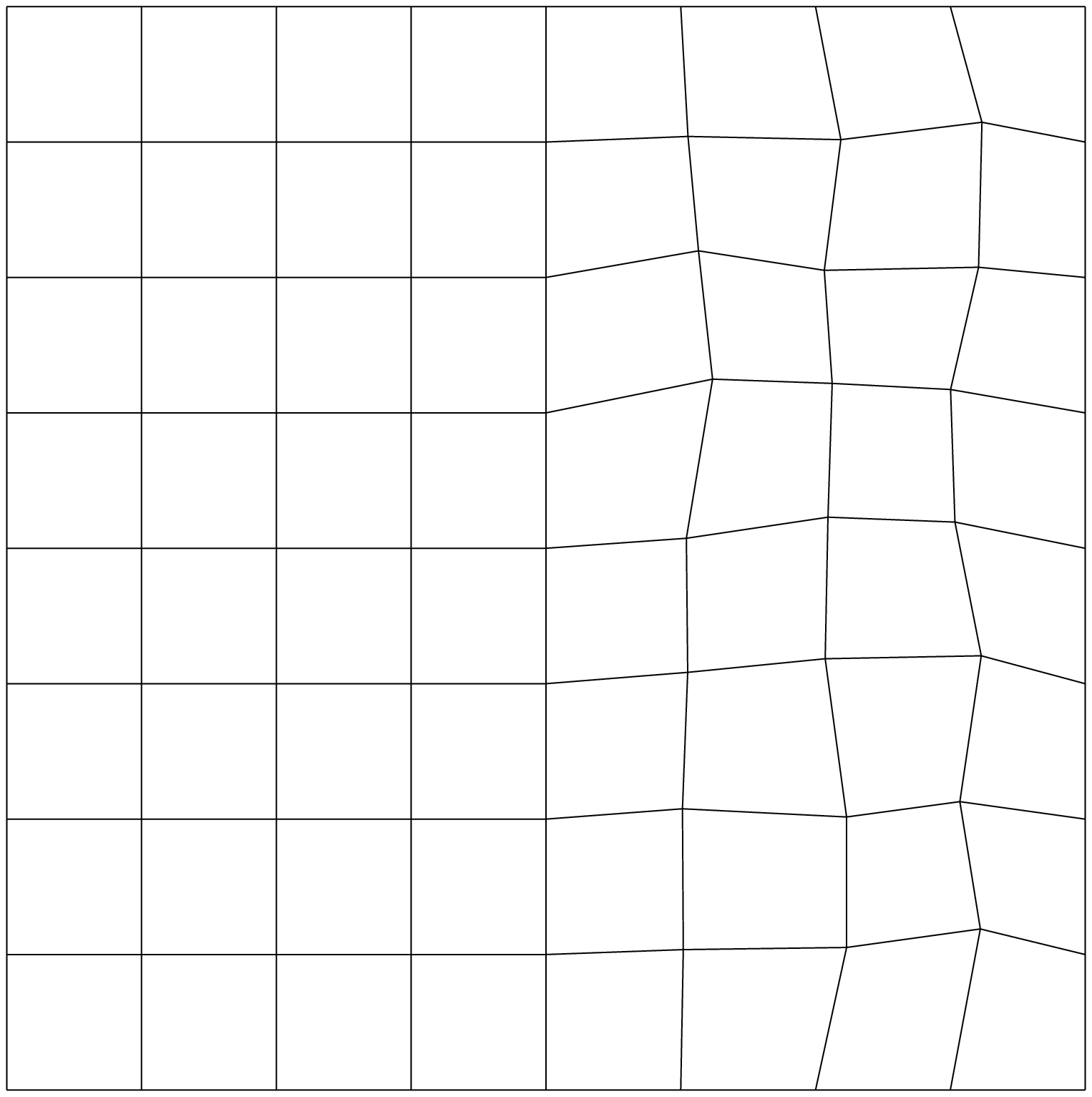}\\
                          \centering\includegraphics[height=3cm, width=3cm]{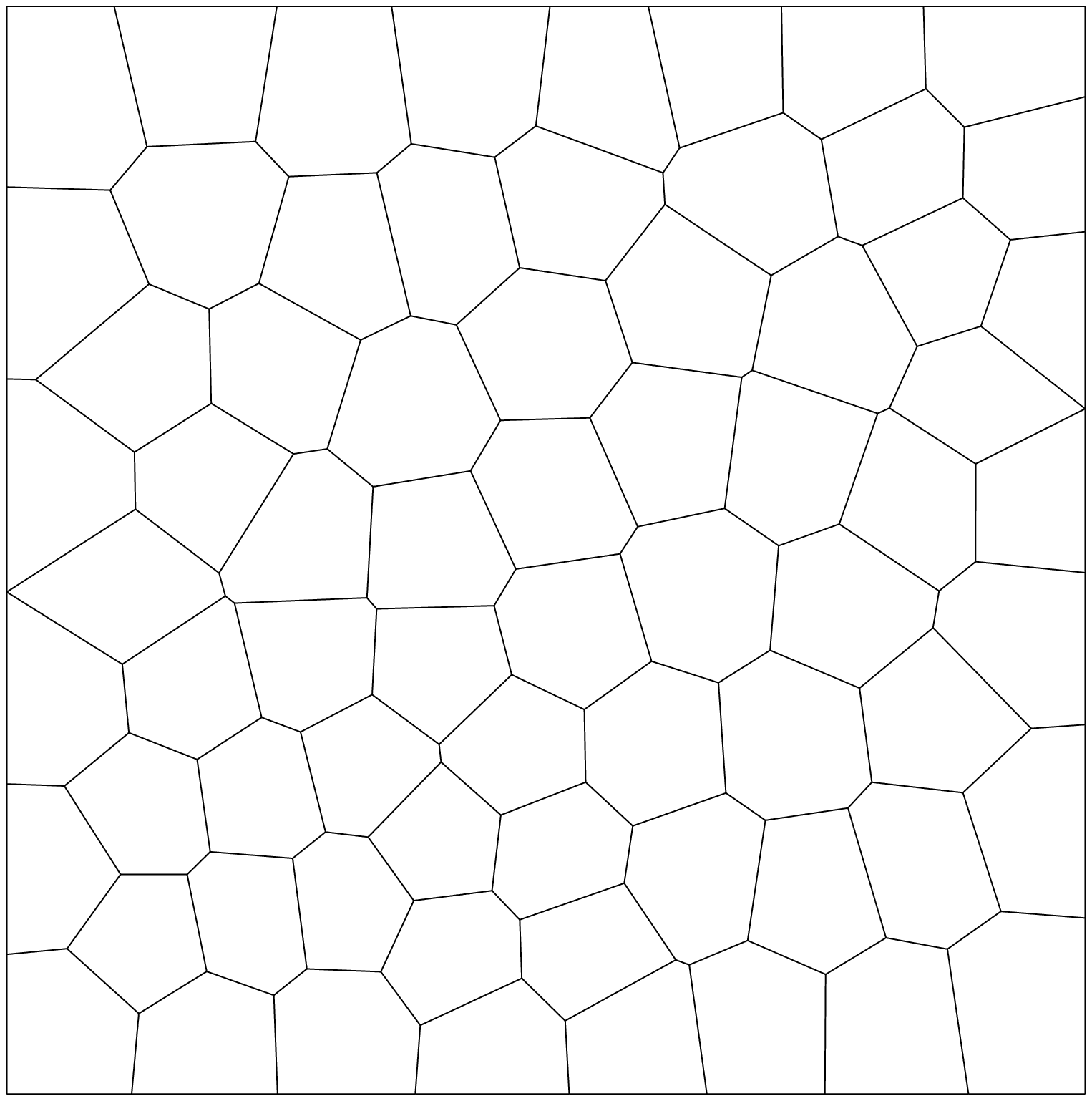}
                          \centering\includegraphics[height=3cm, width=3cm]{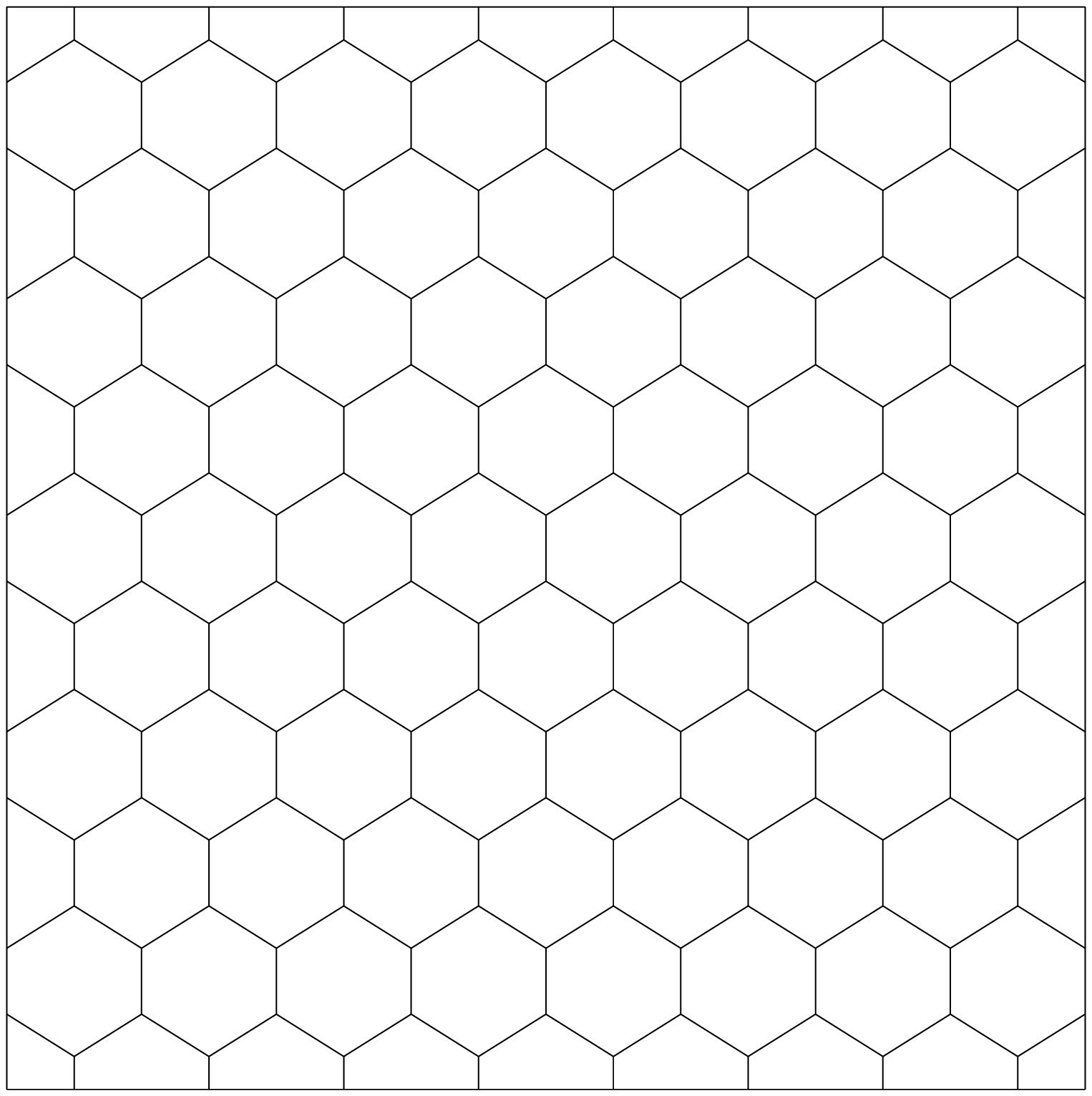}
                          \centering\includegraphics[height=3cm, width=3cm]{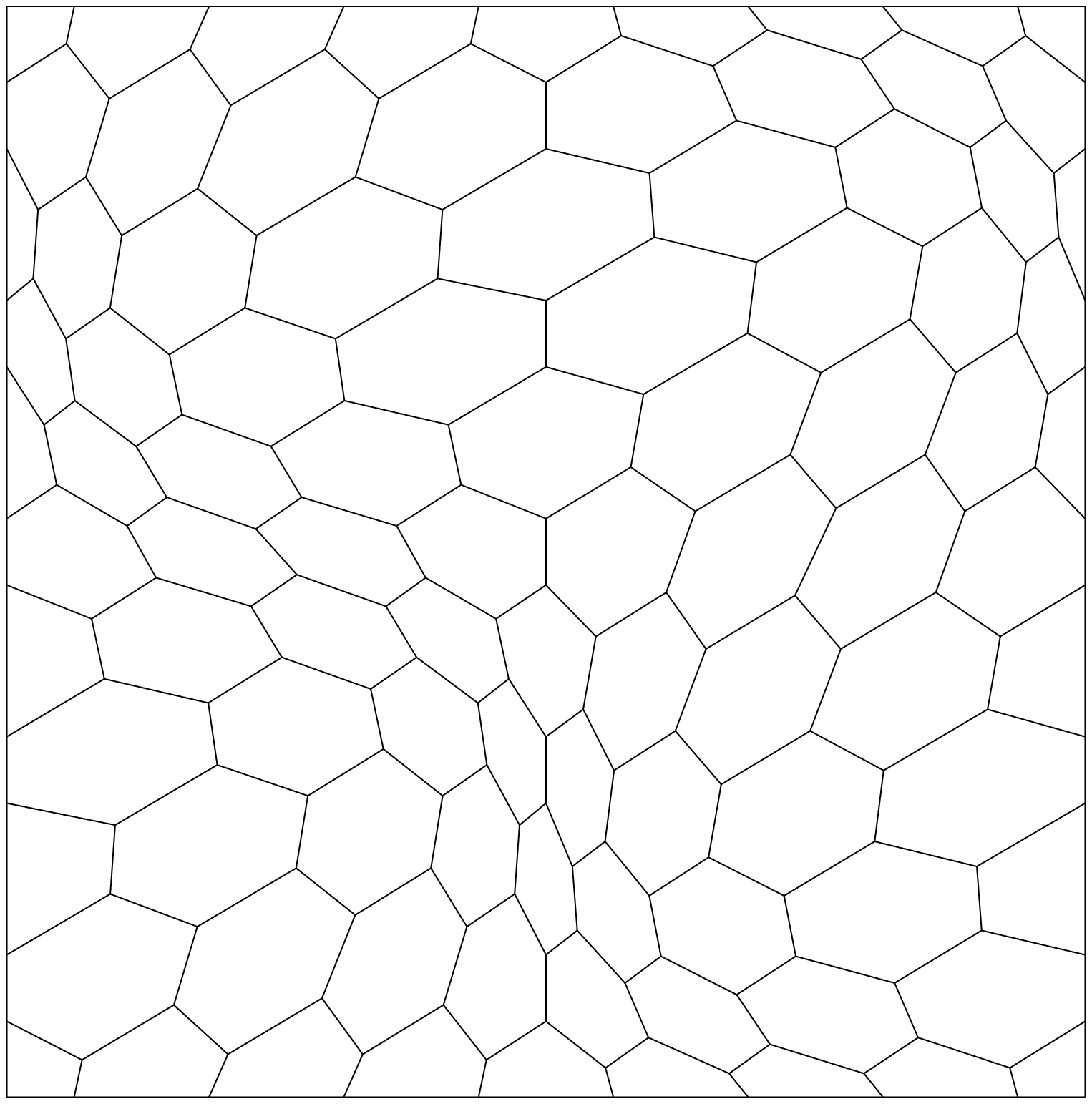}
                   \end{minipage}
		\caption{Sample meshes. From top left to bottom right: $\CT_h^1$, $\CT_h^2$, $\CT_h^3$, $\CT_h^4$, $\CT_h^5$, $\CT_h^6$ y $\CT_h^7$ respectively, with $N=8.$ 
}
		\label{FIG:meshes}
	\end{center}
\end{figure}
%We observe that $\CT_h^1$ is the classic mesh of triangles used in FEM. On the other hand, $\CT_h^2$, $\CT_h^3$ and $\CT_h^4$ are cuadrilateral meshes where we consider, respectively, symmetric squares, deformed quadrilaterals and quadrilaterals that violates the symmetry with respect to the $y$ axis. 
\subsection{Unit square domain with mixed boundary conditions.} We begin with the unit square $\O:=(0,1)^2$ as computational domain.
For this test, we consider the mixed boundary conditions of problem \eqref{system2}. More precisely, we will fix only the bottom of the square  which corresponds to the side with extreme points $(0,0)$ and $(1,0)$. 

In Table \ref{TABLA:1} we report the first six computed eigenvalues with our method. 

\begin{table}[H]
\begin{center}
\caption{Test 1. Computed lowest eigenvalues $\l_{h}^{(i)}$, $1\le
i\le6$, on different meshes.}
\begin{tabular}{|c|c|cccc|c|c|c| c}
\hline
$\CT_h$   & $\l_{h}^{(i)}$ & $N=30$ & $N=40$ & $N=50$ & $N=60$ & 
Order & Extr.&\cite{MMR3} \\
\hline
          & $\l_{h}^{(1)}$ &2.4668 &  2.4671 &   2.4672  &  2.4673 &   2.05 &   2.4674         &2.4674\\
$\CT_h^1$& $\l_{h}^{(2)}$ & 6.2673 &  6.2726 &   6.2751 &   6.2763   & 2.09 &   6.2791&6.2799  \\
          & $\l_{h}^{(3)}$&15.1721  & 15.1881 &  15.1958 &  15.1998   & 1.94  & 15.2096  &15.2090\\
          & $\l_{h}^{(4)}$&22.1607 &  22.1806  & 22.1899 &  22.1949  &  1.99  & 22.2064  &22.2065\\
         & $\l_{h}^{(5)}$&26.8589  & 26.8963  & 26.9158  & 26.9253 &   1.80  & 26.9525   &26.9479\\
& $\l_{h}^{(6)}$&42.9514  & 43.0348 &  43.0726  & 43.0933 &   2.05  & 43.1384          &43.1419\\
   \hline
          & $\l_{h}^{(1)}$ &  2.4651   & 2.4661 &   2.4666 &   2.4668 &   2.12  &  2.4673&2.4674 \\
$\CT_h^{2}$& $\l_{h}^{(2)}$ & 6.2242   & 6.2474   & 6.2586   & 6.2647  &  1.88 &   6.2798&6.2799 \\
          & $\l_{h}^{(3)}$&15.1023  & 15.1464 &  15.1679  & 15.1800  &  1.81  & 15.2110&15.2090\\
          & $\l_{h}^{(4)}$&  22.0254  & 22.1043 &  22.1410 &  22.1610   & 1.98 &  22.2071&22.2065\\
         & $\l_{h}^{(5)}$  & 26.6118 &  26.7567 &  26.8247 &  26.8621  &  1.95 &  26.9494&26.9479\\
         & $\l_{h}^{(6)}$ & 42.5313  & 42.7929  & 42.9164  & 42.9842   & 1.93  & 43.1454&43.1419\\
\hline
  & $\l_{h}^{(1)}$ &2.4652  &  2.4662 &   2.4666  &  2.4669 &   1.97  &  2.4675&2.4674\\
   $\CT_h^{3}$& $\l_{h}^{(2)}$ &6.2236 &   6.2470 &   6.2585  &  6.2645  &  1.88  &  6.2799&6.2799\\
    & $\l_{h}^{(3)}$ & 15.0980 &  15.1435  & 15.1665  & 15.1786   & 1.79 &  15.2117&15.2090\\
    & $\l_{h}^{(4)}$ & 22.0324 &  22.1077  & 22.1433  & 22.1624   & 1.96 &  22.2075&22.2065\\
    & $\l_{h}^{(5)}$ & 26.6130 &  26.7571 &  26.8256 &  26.8621  &  1.95 &  26.9494&26.9479\\
    & $\l_{h}^{(6)}$ & 42.5402  & 42.7955 &  42.9189  & 42.9852  &  1.89 &  43.1503&43.1419\\
   \hline
 \end{tabular}
\label{TABLA:1}
\end{center}
\end{table}

Clearly from Table \ref{TABLA:1} we observe that the optimal quadratic  order of approximation of the eigenvalues is obtained, as is expected according to Theorem \ref{thm:double:order}. Moreover, the computed extrapolated values are close to those computed with the BDM elements in \cite{MMR3}.

\subsection{Rigid square domain.} 
In the following examples, we will consider $\bu=\boldsymbol{0}$ as boundary condition for the whole domain. This  leads to the fact that, for the implementation of the eigenvalue problem, the condition $\int_{\O}p_h=0$ must be incorporated in the matrix system as a Lagrange multiplier. Clearly this condition is equivalent to impose $\int_{\O}\tr(\bsig_h)=0$ and its computation is based in \eqref{eq:trace_comput}.

Indeed, the term $\int_E\tr(\boldsymbol{\bsig_h})$ is computable according to \eqref{eq:global_space} since, for  $\boldsymbol{\xi}\in\bcW_h^E$ we have
\begin{equation}
\label{eq:trace_comput}
\ds\int_E\tr(\boldsymbol{\bxi})=\int_E\boldsymbol{\xi}:\mathbb{I}=\int_E\boldsymbol{\xi}:\nabla\boldsymbol{x}=-\int_E\boldsymbol{x}\cdot\bdiv\boldsymbol{\xi}+\int_{\partial E}\boldsymbol{\xi}\boldsymbol{n}\cdot\boldsymbol{x},
\end{equation}
where $\boldsymbol{x}\in\mathbf{P}_k(E)$.

For this test we consider the square $\O:=(-1,1)^2$ as computational domain. As we claim above, the boundary condition in this test is $\bu=\0$ in the whole boundary. In Table \ref{TABLA:2} we present the obtained results with the VEM method.

%\vspace*{-\baselineskip}
\begin{table}[H]
\begin{center}
\caption{Test 2. Computed lowest eigenvalues $\l_{h}^{(i)}$, $1\le
i\le5$, on different meshes.}
\begin{tabular}{|c|c|cccc|c|c|c|}
\hline
$\CT_h$   & $\l_{h}^{(i)}$ & $N=30$ & $N=40$ & $N=50$ & $N=60$ &
Order & Extr. &\cite{lovadina}\\
\hline
          & $\l_{h}^{(1)}$ &13.0092  & 13.0435  & 13.0583 &  13.0669 &   2.12 &  13.0839 & 13.086 \\
$\CT_h^1$& $\l_{h}^{(2)}$ & 22.7920 &  22.8983 & 22.9456 &  22.9697 &   2.19  & 23.0198 & 23.031 \\
          & $\l_{h}^{(3)}$&22.7961 &  22.8985  & 22.9457  & 22.9698 &   2.09 &  23.0234&23.031\\
          & $\l_{h}^{(4)}$&31.5769 &  31.7916 &  31.8819 &  31.9329  &  2.23  & 32.0281& 32.053\\
         & $\l_{h}^{(5)}$&37.8846  & 38.1650  & 38.2970 &  38.3681   & 1.96  & 38.5358& 38.532\\

\hline
          & $\l_{h}^{(1)}$ & 12.8975  & 12.9789  & 13.0171 &  13.0381   & 1.95 &  13.0872& 13.086 \\
$\CT_h^4$& $\l_{h}^{(2)}$ &22.2721&   22.5976 &  22.7517 &  22.8365 &   1.92 &  23.0393 & 23.031 \\
          & $\l_{h}^{(3)}$&22.2768  & 22.5996  & 22.7529  & 22.8371   & 1.91  & 23.0405 &23.031\\
          & $\l_{h}^{(4)}$&  30.8797 &  31.3801  & 31.6183 &  31.7492   & 1.91 &  32.0646 & 32.053\\
         & $\l_{h}^{(5)}$  & 36.2345  & 37.2064 &  37.6732  & 37.9316  &  1.87  & 38.5714& 38.532\\
\hline
& $\l_{h}^{(1)}$ &12.9192  & 12.9953 &  13.0273 &  13.0498 &   1.92 &  13.0953 &13.086 \\
$\CT_h^5$& $\l_{h}^{2}$  & 22.5009 &  22.7472 &  22.8523 &  22.9142 &   2.13 &  23.0347 & 23.031 \\
& $\l_{h}^{(3)}$ &   22.5136 &  22.7527 &  22.8601 &  22.9197 &   2.06 &  23.0472&23.031\\
& $\l_{h}^{(4)}$ &   31.0347 &  31.5018 &  31.6938 &  31.8194 &   2.10 &  32.0511& 32.053\\
& $\l_{h}^{(5)}$ &   37.1240 &  37.7922 &  38.0445 &  38.2329 &   2.16 &  38.5360&38.532\\
\hline
 \end{tabular}
\label{TABLA:2}
\end{center}
\end{table}

Once again, the quadratic order is obtained and the extrapolated values are close to those in \cite{lovadina}. We remark that in \cite{lovadina} the authors have 
considered the classic velocity-pressure formulation for the Stokes eigenvalue problem, which is clearly less expensive than the pseudostress formulation of \cite{MMR3}. However, since in our case we are not considering the mixed formulation, the $\texttt{eigs}$ solver of MATLAB works perfectly,
thanks to  the elliptic formulation. 
%Also in \cite{LM}, the DG method shows a very good accuracy for the approximation for the eigenvalues for this example with the pseudostress formulation. However, 
%our discrete formulation does not need additional terms to compensate the jumps of the normal component between two elements as it happens in the DG setting, since our method is $\mathbb{H}(\bdiv)$ conforming. 

In Figure \ref{FIG:pressureThT} we present plots for the first and fourth eigenfunctions. This plots show the magnitude of the velocity and the corresponding 
vector field. For the first eigenfunction we present plots obtained with a triangular mesh and for the fourth eigenfunction plots obtained with a Voronoi mesh.

\begin{figure}[H]
	\begin{center}
		\begin{minipage}{13cm}
			\centering\includegraphics[height=6cm, width=6cm]{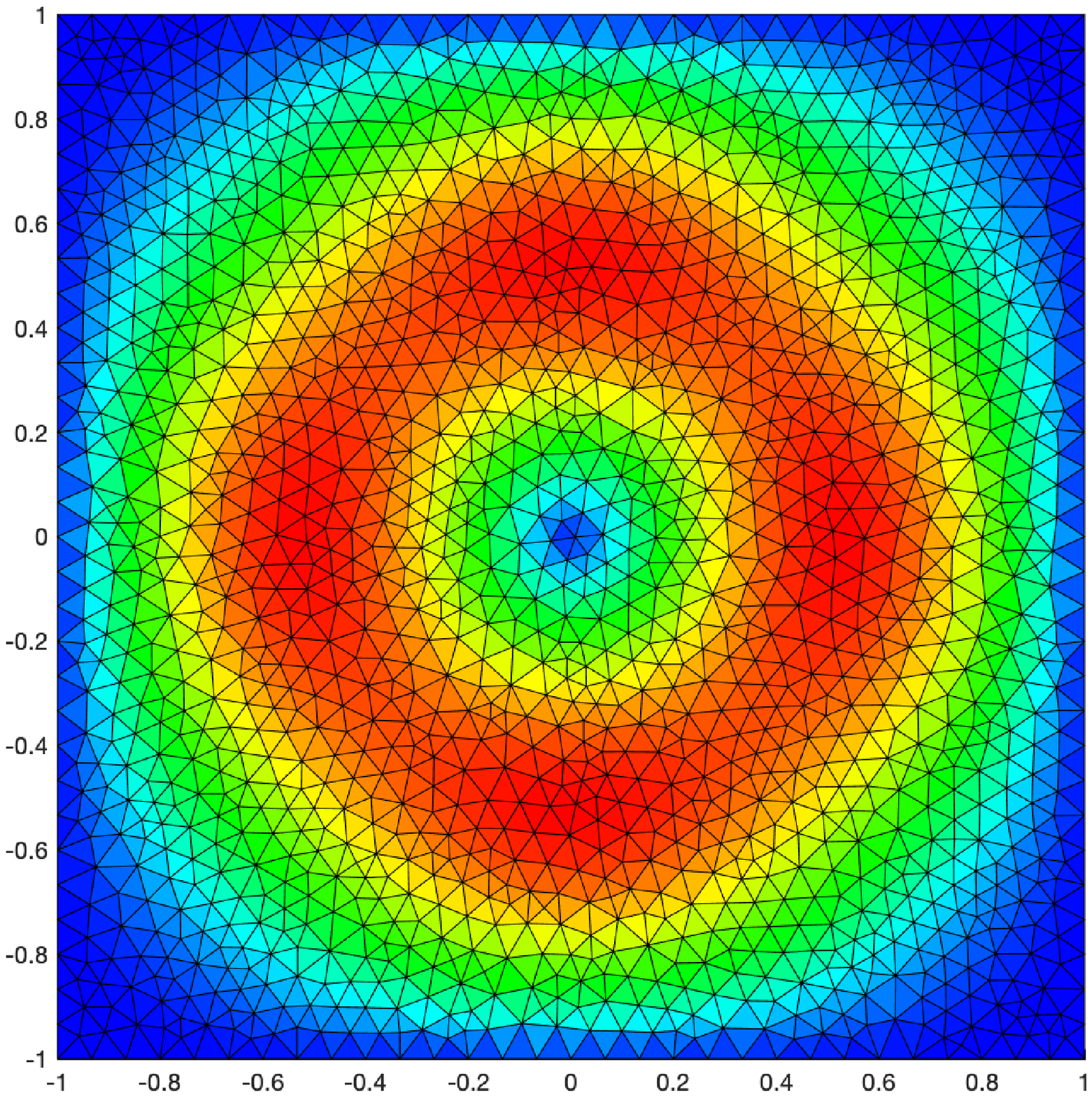}
			\centering\includegraphics[height=6cm, width=6cm]{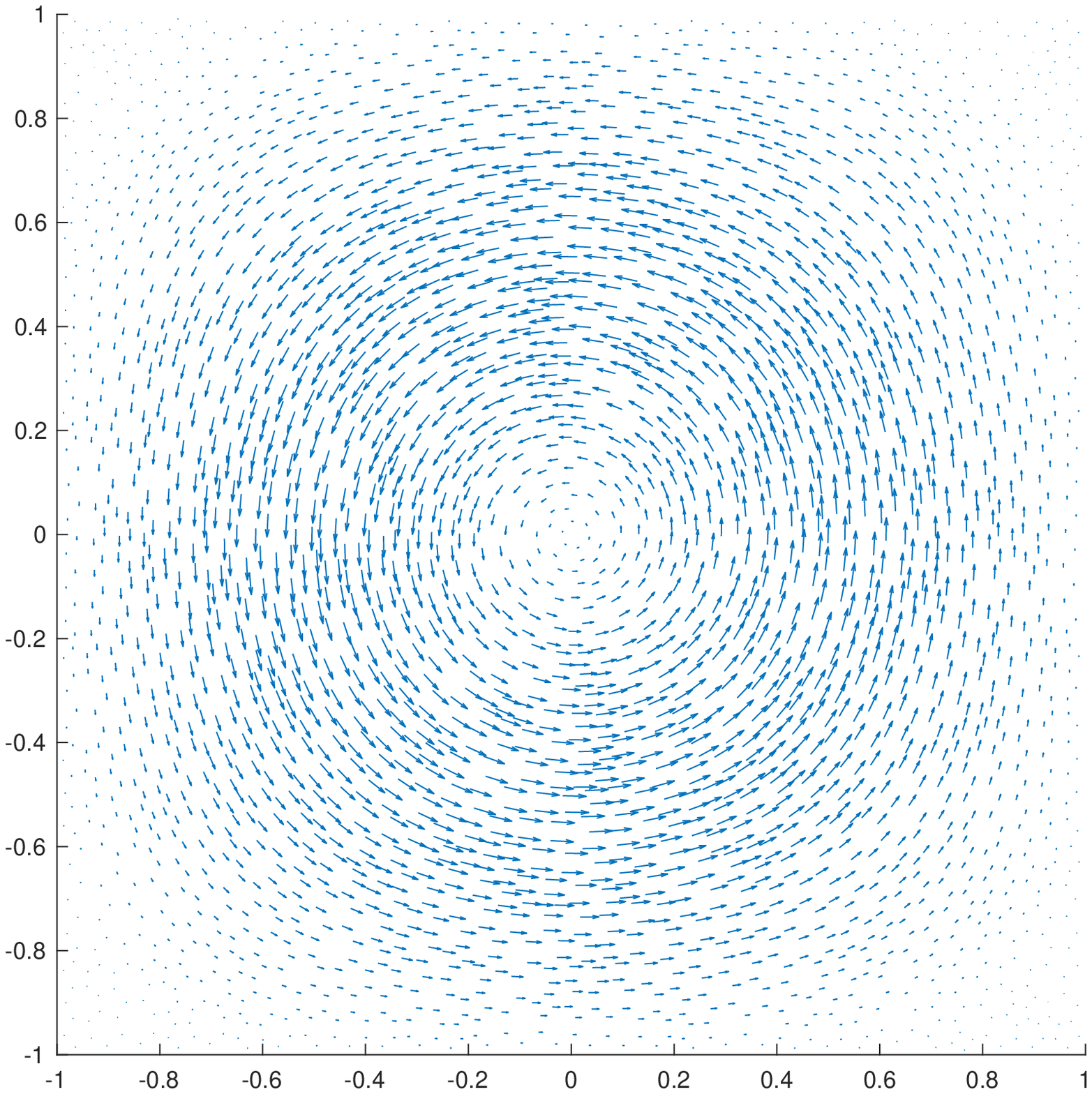}\\
                          \centering\includegraphics[height=6cm, width=6cm]{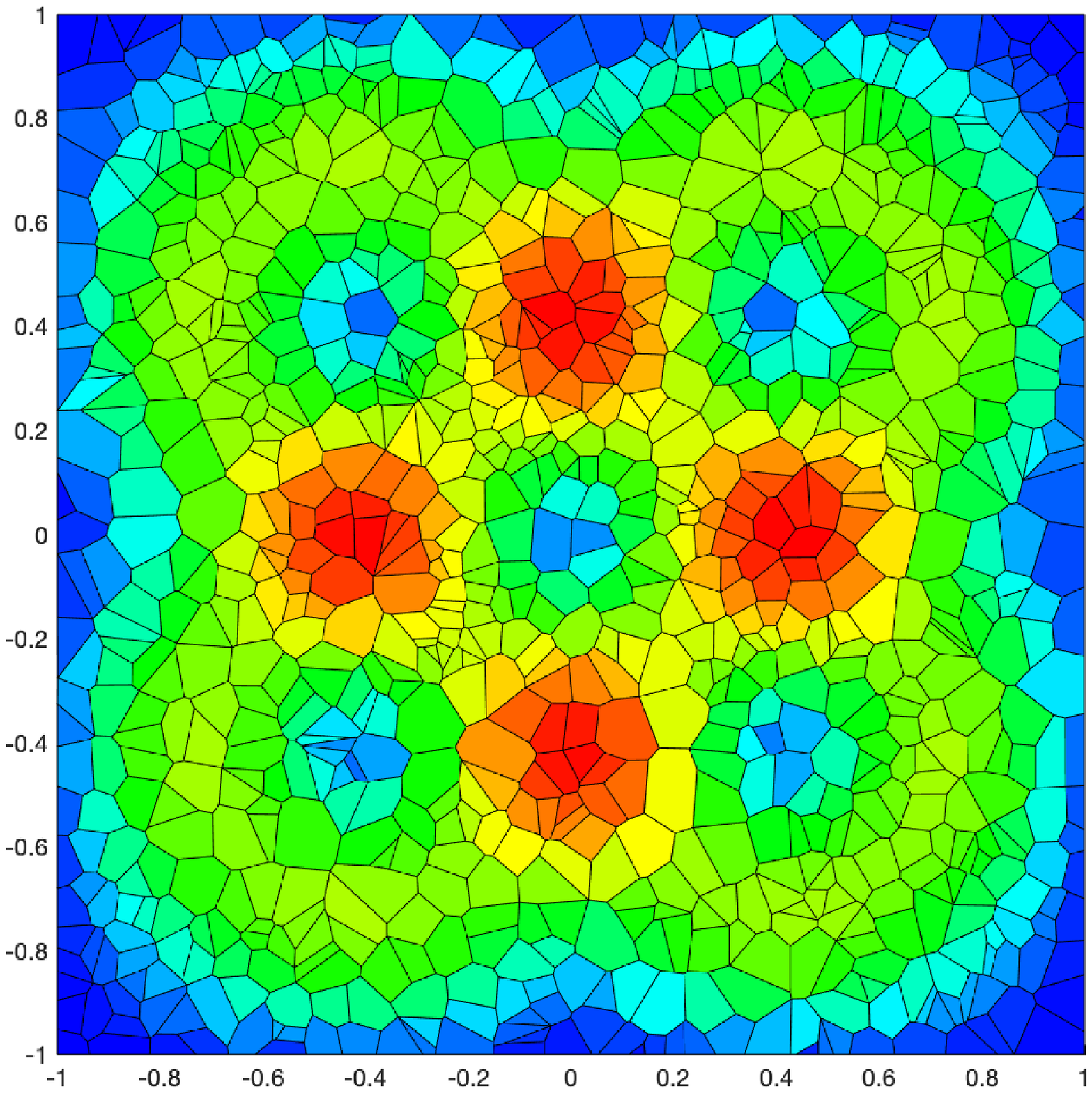}
                          \centering\includegraphics[height=6cm, width=6cm]{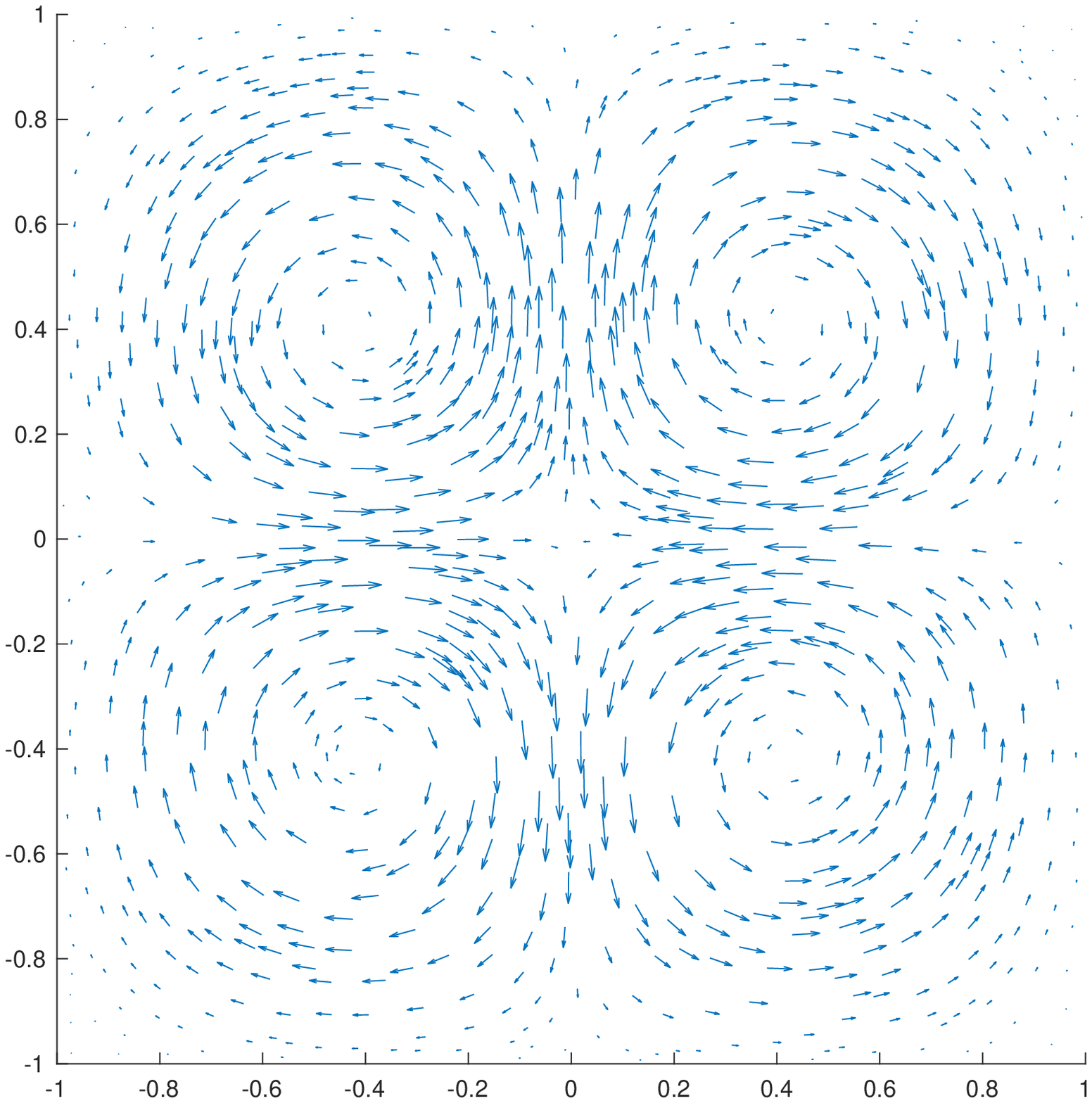}
                   \end{minipage}
		\caption{Eigenfunctions corresponding to the first and fourth lowest eigenvalues with $\CT_{h}^{1}$ and $\CT_{h}^{5}$: magnitude of  $\bu_{h}^{1}$ (upper left) velocity field of $ \bu_{h}^{1}$ (upper right),  magnitude of  $\bu_{h}^{4}$ (bottom left) and velocity field $ \bu_{h}^{4}$ (bottom right).
}
		\label{FIG:pressureThT}
	\end{center}
\end{figure}
\subsection{L-shaped domain.} In this test we consider a non-convex domain that we call the L-shaped domain, which is defined by $\O_{L}:=(-1,1)\times(-1,1)\backslash[-1,0]\times[-1,0]$.  In this case, the optimal order is not expectable for the eigenfunctions, due the presence of the singularity in $(0,0)$. In fact, the rate $r$ of convergence for the eigenvalues is such that $1.7\leq r\leq 2$, depending on the regularity of the eigenfunctions. In the following table we report the results for this configuration of the problem.

\begin{table}[H]
\begin{center}
\caption{Test 3. Computed lowest eigenvalues $\l_{h}^{(i)}$, $1\le
i\le5$, on different meshes.}
\begin{tabular}{|c|c|cccc|c|c|c|}
\hline
$\CT_h$   & $\l_{h}^{(i)}$  & $N=19$ & $N=27$ & $N=35$ & $N=45$ &
Order & Extr. &\cite{lovadina}\\
\hline
 & $\l_{h}^{(1)}$         & 31.1821 &  31.5813  & 31.7561 &  31.8593  &  1.76 &  32.0506 & 32.1734\\
 $\CT_h^6$& $\l_{h}^{(2)}$ & 36.2530 &  36.6458  & 36.7964 &  36.8751  &  2.19 &  36.9872& 37.0199\\
& $\l_{h}^{(3)}$   & 41.1780 &  41.5727  & 41.7223 &  41.8026  &  2.19 &  41.9146& 41.9443\\
& $\l_{h}^{(4)}$   & 47.9143 &  48.4647  & 48.6773 &  48.7953  &  2.11 &  48.9664& 48.9844\\
& $\l_{h}^{(5)}$   & 53.8827 &  54.6302  & 54.9298 &  55.1019  &  1.99 &  55.3698& 55.4365\\
& $\l_{h}^{(6)}$   & 67.1556 &  68.2905  & 68.7424 &  68.9892   & 2.06 &  69.3656& 69.5600\\
         \hline
& $\l_{h}^{(1)}$  & 31.0337  & 31.5027 &  31.7066 &  31.8259 &   1.78 &  32.0452& 32.1734\\
  $\CT_h^7$& $\l_{h}^{(2)}$  & 36.0658 &  36.5552 &  36.7432 &  36.8405 &   2.19 &  36.9804 & 37.0199\\
& $\l_{h}^{(3)}$   & 41.0445 &  41.5115 &  41.6874 &  41.7795 &   2.23 &  41.9064 & 41.9443\\
& $\l_{h}^{(4)}$   & 47.7733 &  48.3911 &  48.6319 &  48.7653 &   2.09 &  48.9622 & 48.9844\\
& $\l_{h}^{(5)}$   & 53.7141 &  54.5342 &  54.8676 &  55.0610 &   1.95 &  55.3695 & 55.4365\\
& $\l_{h}^{(6)}$   & 66.9808 &  68.1944 &  68.6820 &  68.9492 &   2.03 &  69.3668 & 69.5600\\
\hline
 \end{tabular}
\label{TABLA:4}
\end{center}
\end{table}

We observe that for the first eigenvalue,  the order of
approximation is not optimal. However, this order is the expected since the eigenfunctions associated 
to this eigenvalue are singular due the non convexity of the geometry at the point $(0,0)$, leading to
a lack of regularity of the eigenfunction and hence, a poorer convergence order. However, for the 
rest of the eigenvalues the approximation order is quadratic precisely because the associated eigenfunctions
to these eigenvalues are more regular. We remark that for other polygonal meshes the results are similar.

Finally, in Figure \ref{FIG:plotsL} we present plots of the magnitude and velocity fields for the first  
and  second eigenfunctions, obtained with hexagonal and deformed hexagonal meshes, respectively.
\begin{figure}[H]
	\begin{center}
		\begin{minipage}{13cm}
			\centering\includegraphics[height=6cm, width=6cm]{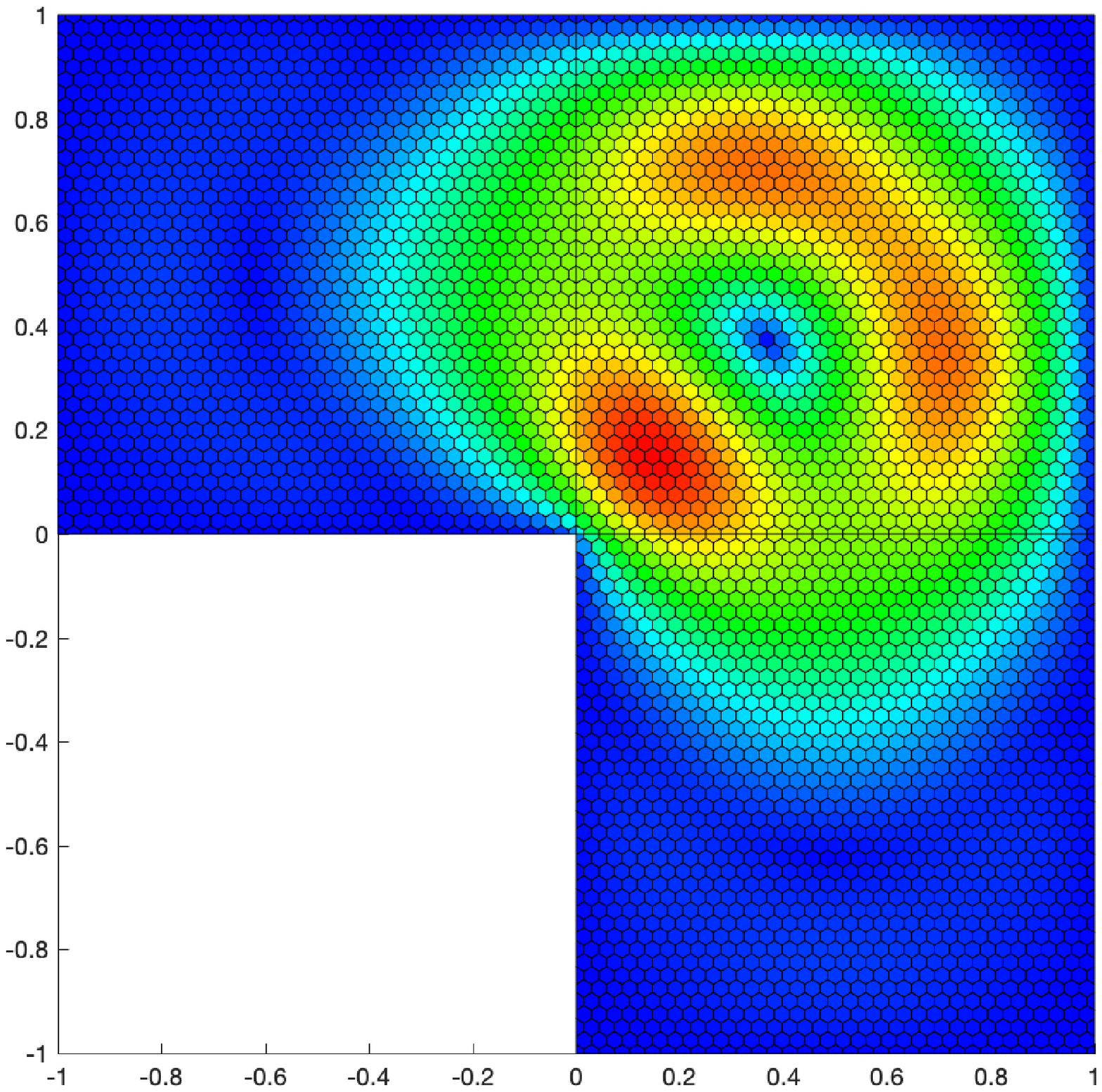}
			\centering\includegraphics[height=6cm, width=6cm]{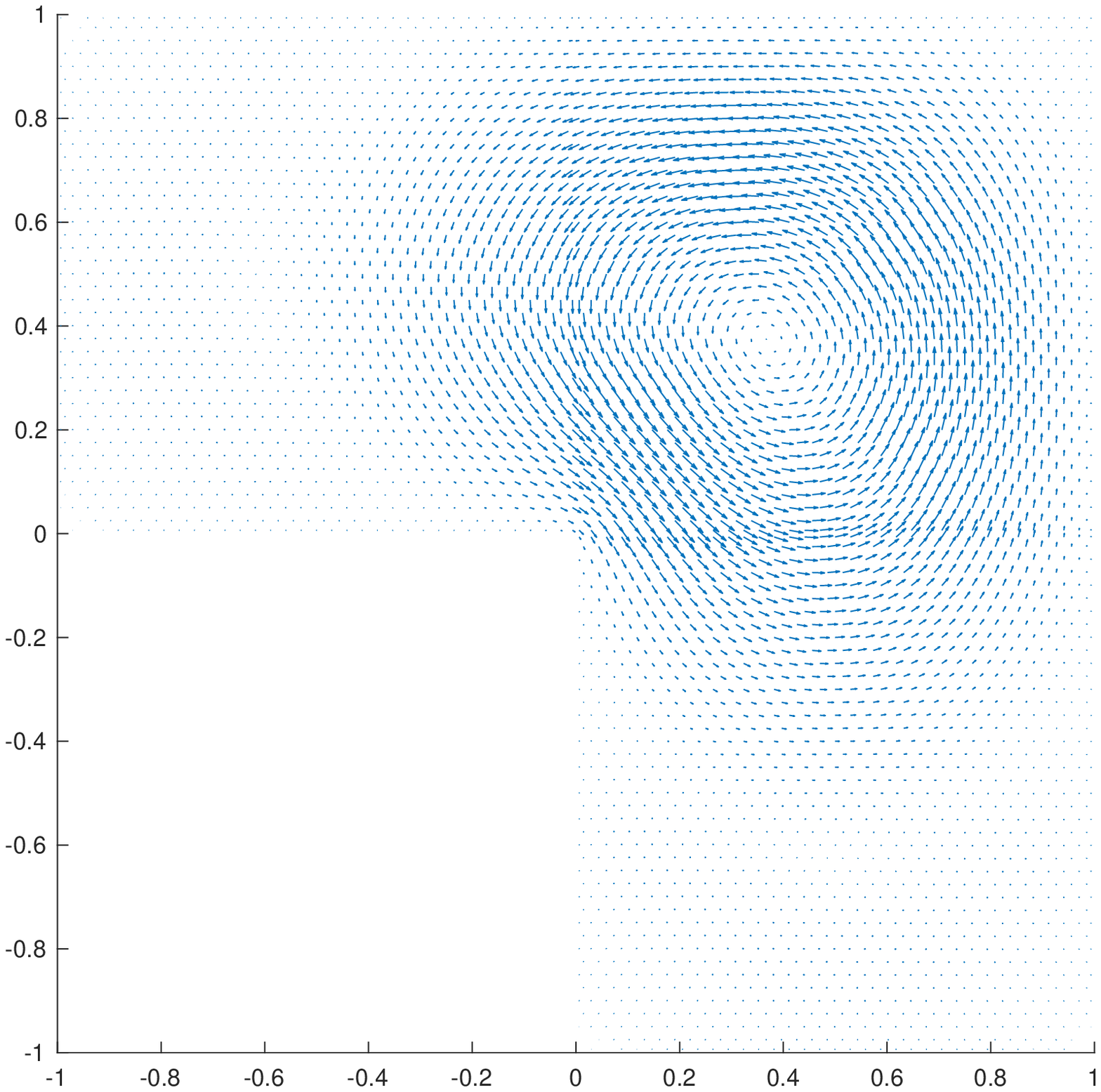}\\
                          \centering\includegraphics[height=6cm, width=6cm]{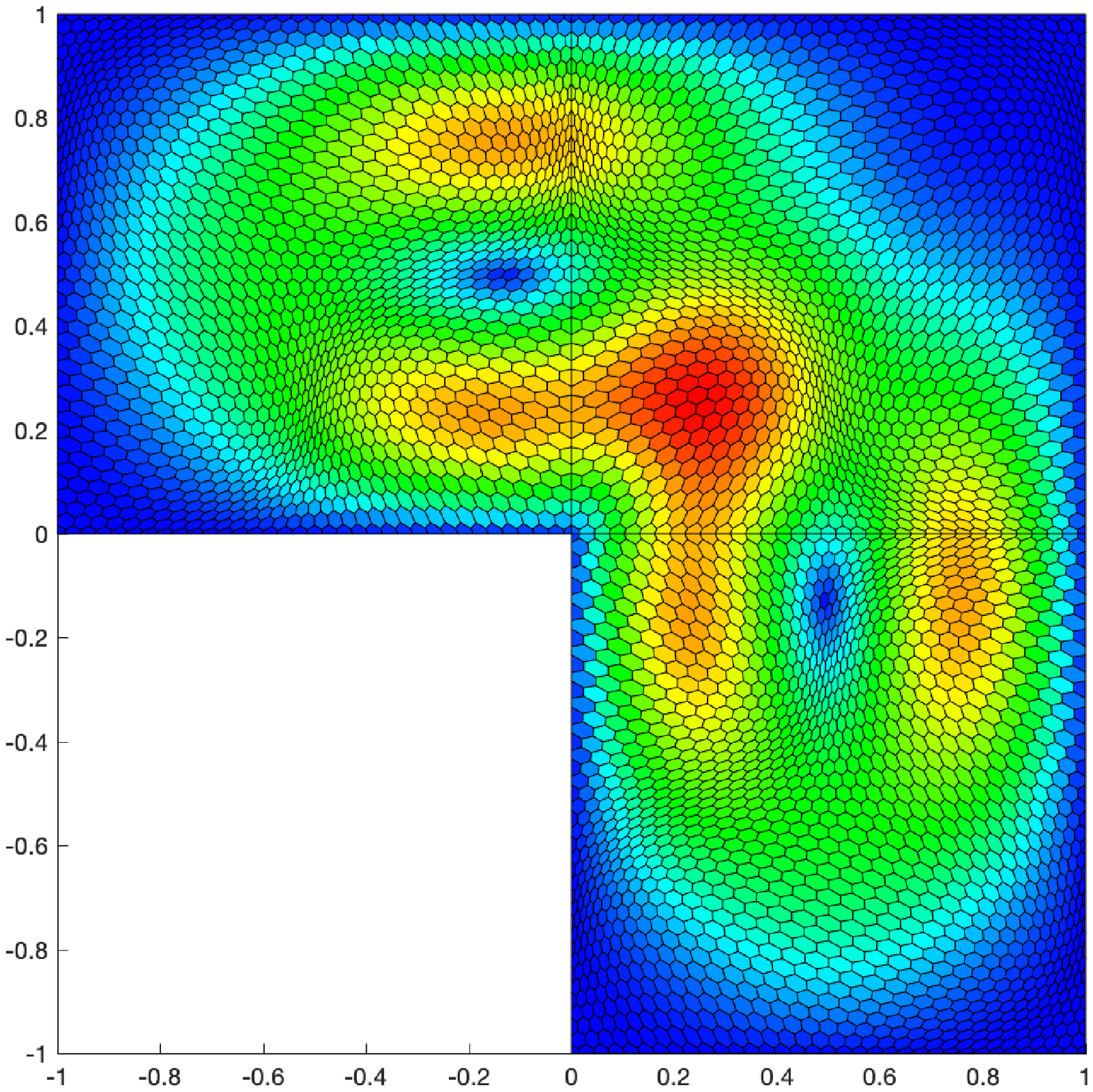}
                          \centering\includegraphics[height=6cm, width=6cm]{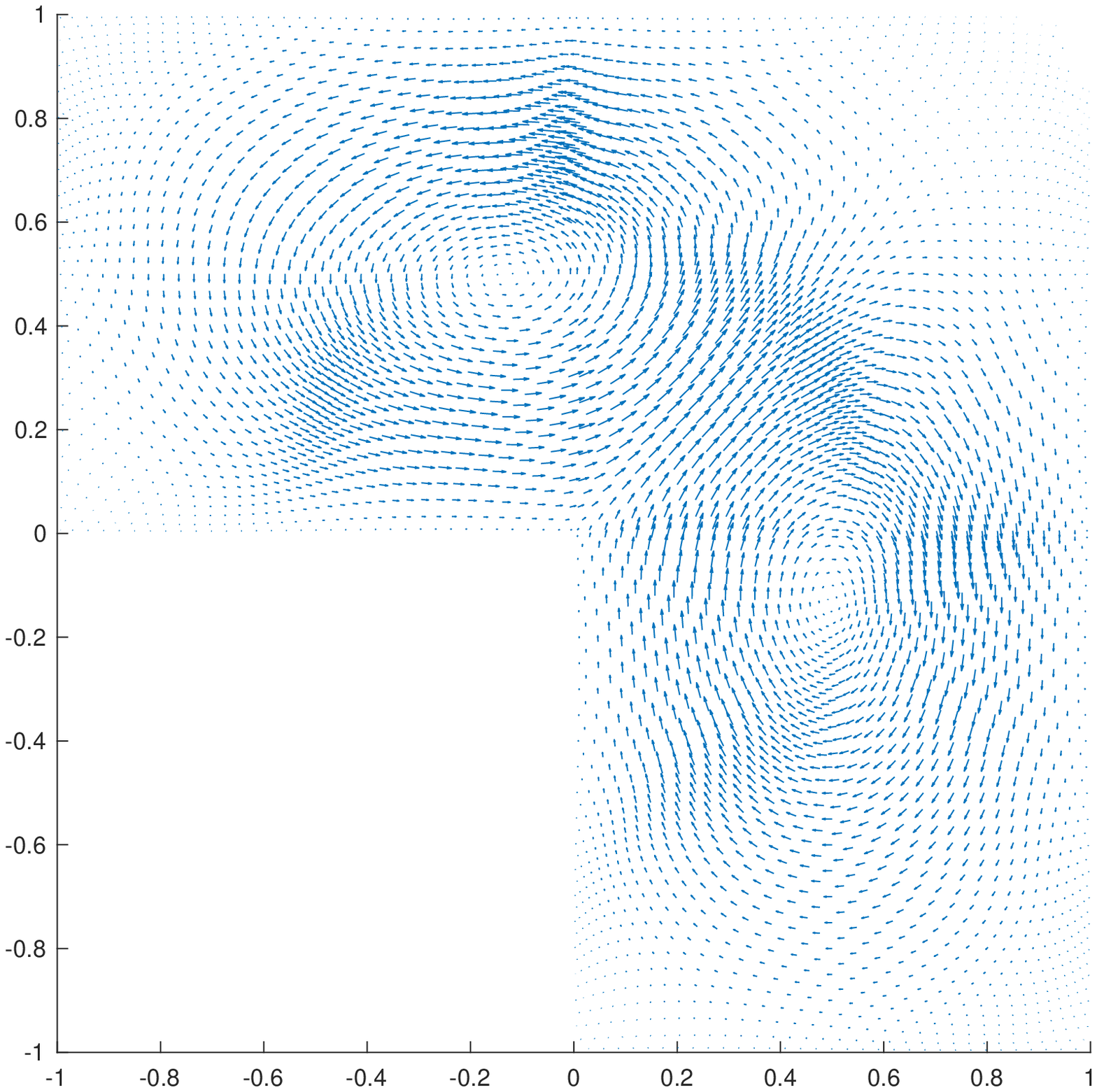}
                   \end{minipage}
		\caption{Eigenfunctions corresponding to the first and second  lowest eigenvalues with $\CT_{h}^{6}$ and $\CT_{h}^{7}$: magnitude of  $\bu_{h}^{1}$ (top left), velocity field of $ \bu_{h}^{1}$  (top right),  magnitude of  $\bu_{h}^{2}$ (bottom left), velocity field of $ \bu_{h}^{2}$ (bottom right).
}
		\label{FIG:plotsL}
	\end{center}
\end{figure}
%%%%%%%%%%%%%%%%%%%%%%%%%%%
\section{Aknowledgments}

The authors are deeply grateful to Prof. Rodolfo Rodr\'iguez (Universidad de Concepci\'on, Chile)
for the comments and observations which improved the manuscript.

\bibliographystyle{siam} %siam, abbrv, ieeetr, unsrt, acm
\bibliography{references}
\bibliographystyle{plain}

\end{document}